\documentclass[a4paper, 11pt]{amsart}

\usepackage[english]{babel}
\usepackage[utf8x]{inputenc}
\usepackage[T1]{fontenc}
\usepackage{amsmath}
\usepackage{amssymb}
\usepackage{amsthm}
\usepackage{stmaryrd}
\usepackage[pagewise, displaymath, mathlines]{lineno}

\usepackage[a4paper,top=3cm,bottom=2cm,left=2cm,right=2cm,marginparwidth=1.75cm]{geometry}

\usepackage{array}
\usepackage{enumerate}
\usepackage{amsmath}
\usepackage{graphicx, caption}
\usepackage[colorinlistoftodos]{todonotes}
\usepackage[colorlinks=true, allcolors=blue]{hyperref}
\usepackage{mathtools}

\newtheorem{thm}{Theorem}

\newtheorem{prop}[thm]{Proposition}
\newtheorem{coroll}[thm]{Corollary}
\newtheorem{rmk}[thm]{Remark}
\newtheorem{lemma}[thm]{Lemma}
\newtheorem{hypo}{Hypothesis}

\newcommand\blfootnote[1]{%
	\begingroup
	\renewcommand\thefootnote{}\footnote{#1}%
	\addtocounter{footnote}{-1}%
	\endgroup
}

\newcommand{\R}{\mathbb{R}}
\newcommand{\pO}{\partial \Omega}

\newcommand{\EE}{\mathbb{E}}

\newcommand{\loc}{\text{loc}}

\newcommand{\CC}{\mathcal{C}}

\newcommand{\vertiii}[1]{{\left\vert\kern-0.25ex\left\vert\kern-0.25ex\left\vert #1 
		\right\vert\kern-0.25ex\right\vert\kern-0.25ex\right\vert}}

\newcommand{\dd}{\color{black}}

\newcommand{\vp}{v_{\perp}}
\newcommand{\vt}{v_{\parallel}}
\newcommand{\up}{u_{\perp}}
\newcommand{\ut}{u_{\parallel}}

\newcommand{\rp}{r_{\perp}}
\newcommand{\rt}{r_{\parallel}}

\newcommand{\MBC}{\textbf{(MBC)}}
\newcommand{\CLBC}{\textbf{(CLBC)}}
\newcommand{\MBCM}{\textbf{(MBCM)}}
\newcommand{\MBCkappa}{\textbf{(MBC$\kappa$)}}

\renewcommand{\d}{\mathrm{d}}

\title[Degenerate Kinetic Equations with Non-Isothermal Boundaries]{Asymptotic Behavior of Degenerate Linear Kinetic Equations with Non-Isothermal Boundary Conditions}

\author{Armand Bernou}
\address{Dipartimento di Matematica, Universit\`a di Roma ``La Sapienza'', P.le Aldo Moro 2, 00185 Roma, Italy}
\email{armand.bernou@uniroma1.it}

\begin{document}
	
	\begin{abstract}
		We study the degenerate linear Boltzmann equation inside a bounded domain with a generalized diffuse reflection at the boundary and variable temperature, including the Maxwell boundary conditions with the wall Maxwellian or heavy-tailed reflection kernel and the Cercignani-Lampis boundary condition. Our abstract collisional setting applies to the linear BGK model, the relaxation towards a space-dependent steady state, and collision kernels with fat tails. We prove for the first time the existence of a steady state and a rate of convergence towards it without assumptions on the temperature variations. Our results for the Cercignani-Lampis boundary condition make also no hypotheses on the accommodation coefficients. The proven rate is exponential when a control condition on the degeneracy of the collision operator is satisfied, and only polynomial when this assumption is not met, in line with our previous results regarding the free-transport equation. We also provide a precise description of the different convergence rates, including lower bounds, when the steady state is bounded. Our method yields constructive constants.  
	\end{abstract}
	
	\maketitle

	

	\section{Introduction}
	
	\subsection{Model}
	In this article, we study the degenerate linear Boltzmann equation set inside a $C^2$ bounded domain (open, connected) $\Omega \subset \R^d$, $d \in \{2,3\}$, with some  boundary conditions that we detail below. The initial boundary value problem writes 
	\begin{align}
		\label{eq:main_pb}
		\left\{
		\begin{array}{lll}
			&\partial_t f + v \cdot \nabla_x f = \mathcal{C}f, \qquad &\mathrm{in } \quad \R_+ \times G, \\
			&\gamma_- f = K \gamma_+ f, \qquad &\mathrm{on } \quad \R_+ \times \Sigma_-, \\
			&f_{| t = 0} = f_0, \qquad &\mathrm{in } \quad G,
		\end{array}
		\right.
	\end{align}
	with the notations $G := \Omega \times \R^d$, and, denoting by $n_x$ the unit \textbf{outward} normal vector at $x \in \partial \Omega$, 
	\[ \Sigma := \partial \Omega \times \R^d, \qquad \Sigma_{\pm} := \Big\{(x,v) \in \Sigma, \pm (v \cdot n_x) > 0 \Big\}. \]
	In \eqref{eq:main_pb}, the unknown function $f = f(t,x,v)$ is the so-called distribution function. The quantity $f(t,x,v) \d v \d x$ can be understood as the (non-negative) density at time $t$ of particles whose positions are close to $x$ and velocities close to $v$. We will study \eqref{eq:main_pb} in a $L^1$ framework, and we denote by $\gamma_\pm f$ the trace of $f$ on $\Sigma_\pm$. 
	
	\subsection{The collision operator}
	
	We consider the linear degenerate Boltzmann equation. 
	The corresponding collision operator $\CC$ is defined, for all $f : G \to \R$, for $(x,v) \in \Omega \times \R^d$, by
	\begin{align*}
		\CC f (x,v) = \int_{\R^d} \big( k(x,v',v) f(x,v') - k(x,v,v') f(x,v) \big) \, \d v',
	\end{align*}
	see below the precise assumptions made on the non-negative function $k$ and on $f$ to make sense of this integral.
	The so-called collision kernel, $k$, describes the interactions between the particles and the background. We emphasize that $k$ is modulated in space. Concrete examples of $k$, including the linear BGK model, a heavy-tailed collision kernel, and the (non-degenerate) linear Boltzmann model, are presented in Section \ref{sec:applications}. We may split this collision operator, as 
	\[ \CC f (x,v) = \CC_+ f (x,v) + \CC_- f (x,v), \]
	where the \textit{gain} and \textit{loss} terms are given respectively by 
	\begin{align*}
		\CC_+ f (x,v) = \int_{\R^d} k(x,v',v) f(x,v') \, \d v', \qquad \CC_- f (x,v) = - \Big( \int_{\R^d} k(x,v,v') \, \d v' \Big) f(x,v). 
	\end{align*}
	
	By symmetry, note that the following equality formally holds
	\begin{align}
		\label{eq:sym_k}
		\forall x \in \Omega, \qquad \int_{\R^d} \int_{\R^d} \big( k(x,v',v) f(x,v') - k(x,v,v') f(x,v) \big) \, \d v' \, \d v = 0. 
	\end{align}
	

	\subsection{Boundary conditions}
	
	In this paper we model the interaction between the particles and the wall boundary by generalized diffuse boundary conditions of two types, with varying temperature. Set, for all $x \in \pO$,
	\[ \Sigma_{\pm}^x  := \big\{v \in \R^d, (x,v) \in \Sigma_{\pm} \big\}. \]  
	The boundary operator $K$ is defined, for $\phi$ supported on $(0,\infty) \times \Sigma_+$, for $(t,x,v) \in (0,\infty) \times \Sigma_-$ and assuming that $\phi(t,x,\cdot) \in L^1(\Sigma_+^x,  R(v' \to v;x) |v' \cdot n_x| \d v')$, by
	\begin{align}
		\label{eq:def_K}
		K \phi(t,x,v) = \int_{\Sigma_+^x} \phi(t,x,u) \, R(u \to v;x) \, |u \cdot n_x| \, \d u,
	\end{align}
	with two possible choices for the kernel $R(u \to v;x)$, both depending on some temperature function $\partial\Omega \ni x \to \theta(x)\in \R_+^*$ acting on the boundary and continuous (in the topology induced by local coordinates, as $\Omega$ is $C^2$): 
	\begin{itemize}
		\item \textbf{The Cercignani-Lampis boundary condition \CLBC.} In this case, $R$ is given, for $x$ in $\pO$, $u \in \Sigma_+^x$, $v \in \Sigma_-^x$, by 
		\begin{align}
			\label{eq:def_R}
			R(u \to v; x) &:= \frac{1}{\theta(x) \rp} \frac{1}{(2 \pi \theta(x) \rt (2 - \rt))^{\frac{d-1}{2}}} \exp \Big( -\frac{|\vp|^2}{2 \theta(x) \rp} -\frac{(1-\rp) |\up|^2}{2 \theta(x) \rp} \Big) \\
			&\qquad \times I_0 \Big( \frac{(1-\rp)^{\frac12} \up \cdot \vp}{\theta(x) \rp} \Big) \,  \exp \Big(-\frac{|\vt - (1-\rt) \ut|^2}{2 \theta(x) \rt (2 - \rt)} \Big), \nonumber
		\end{align}
		with the following notations:
		\begin{align*}
			\vp := (v \cdot n_x) n_x, \quad \vt := v - \vp, \quad \up := (u \cdot n_x) n_x, \quad \ut = u - \up,
		\end{align*}
		where $I_0$ is the modified Bessel function given, for all $y \in \R$, by
		\begin{align}
			\label{eq:defI0}
			I_0(y) := \frac{1}{\pi} \int_0^{\pi} \exp \Big( y \cos \phi \Big) \, \d \phi,
		\end{align}
		and where $\theta(x) > 0$ is the wall temperature at $x \in \pO$. The coefficients $\rp \in (0,1]$ and $\rt \in (0,2)$ are the two accommodation coefficients (normal and tangential) at the wall. The value $\vp$ is the normal component of the velocity $v$ at the boundary, while $\vt$ is the tangential component. The same interpretation is of course valid for $u$. 
		
		We will heavily use the normalization property, see \cite[Lemma 10]{Chen_CL_2020}, which, with our notation for $R$, writes, for all $(x,u) \in \Sigma_+$,
		\begin{align}
			\label{eq:normalization_basic}
			\int_{\Sigma_-^x} R(u \to v; x) \, |v \cdot n_x| \, \d v = 1. 
		\end{align}
		Combined with the symmetry property \eqref{eq:sym_k}, this condition will ensure the conservation of mass, as well as the $L^1$ contraction property of the associated semigroup.
		
		\medskip 
		
		\item \textbf{The Maxwell boundary condition \MBC.} For $(x,v) \in \partial \Omega \times \R^d$, we set
		\begin{align}
			\label{eq:def_eta}
			\eta_x(v) := v - 2(v\cdot n_x)n_x.
		\end{align} In this case, $R$ is given, for $x \in \partial \Omega$, $u \in \Sigma_+^x$, $v \in \Sigma_-^x$, by the following formula
		\begin{align}
			\label{eq:R_Maxwell} 
			R(u \to v; x) = \beta(x) M(x,v) + (1-\beta(x)) \delta_{\eta_x(v)}(u) \frac{1}{|v \cdot n_x|},
		\end{align}
		where $\delta_y$ is the Dirac Delta measure at $y \in \R^d$, where $\beta : \partial \Omega \to [0,1]$ is the accommodation coefficient in this setting. The distribution $M$ is defined on $\Sigma_-$ and only depends on $x$ through the temperature $\theta(x)$.  We assume that $M$ satisfies, for all $(x,v) \in \Sigma_-$
		\begin{align}
		\label{eq:property_M}
		M(x,v) > 0, \qquad M(x,v) \le \chi_M, \qquad \int_{\Sigma^x_-} M(x,v') |v' \cdot n_x| \, \d v' = 1,
\end{align}
for some $\chi_M > 0$. Moreover, using also the continuity of $\theta$, we assume that $M$ is continuous on $\Sigma_-$. We finally impose a moment property: there exists $\varepsilon > 0$ such that for all $x \in \partial \Omega$
\begin{align}
    \label{eq:moment_M}
    \int_{\Sigma^x_-} M(x,v) |v|^{1 + \varepsilon} \, \d v < \infty.  
\end{align} 
\end{itemize}

		\subsection{Assumptions and main results}
	
	We denote by $L^p(E; F)$, $1 \le p \le \infty$ the usual $L^p$ spaces of applications from $E$ with values in the Banach space $F$, endowed with the usual norms. We simply write $L^p(E)$ when $F = \mathbb{R}$. We present first our hypotheses regarding the boundary condition:
	\begin{hypo}
		\label{hypo:boundary}
		The boundary operator is defined by \eqref{eq:def_K}, with the reflection operator $R$ given either by
		\begin{enumerate}
			\item case \CLBC: equation \eqref{eq:def_R} with $(\rp, \rt) \in (0,1] \times (0,2)$;
			\item case \MBC: equation \eqref{eq:R_Maxwell} with $\beta \ge \beta_0$ on $\partial \Omega$ for some $0 < \beta_0 \le 1$, with $M$ continuous on $\Sigma_-$ satisfying \eqref{eq:property_M} and \eqref{eq:moment_M}.
		\end{enumerate}
		In both cases, $\theta: \partial \Omega \to \R_+^*$ is continuous. 
	\end{hypo} 
	Regarding the collision operator, we make the following assumptions:
	\begin{hypo}
		\label{hypo:k}
		\begin{enumerate}
			\item $k \in L^{\infty}(\Omega \times \R^d \times \R^d; \R_+)$ with $k_\infty := \sup_{(x,v,v') \in G \times \R^d} |k(x,v,v')|$; 
			\item there exist $\delta_k \in (0, \tfrac12)$, $M_{\delta_k} > 0$ such that for all $x \in \Omega$, $v \in \R^d$, \[\int_{\R^d} k(x,v,v') |v'|^{2\delta_k} \, \d v' \le M_{\delta_k}; \]
			\item there exists $\sigma \in L^{\infty}(\Omega; \R_+)$ such that for all $x \in \Omega$, $v \in \R^d$, $\int_{\R^d} k(x,v,v') \, \d v' = \sigma(x)$. 
		\end{enumerate}
	\end{hypo}
	
	We set $\sigma_\infty := \|\sigma\|_\infty$ in the whole paper.
	To study the long-time behavior of \eqref{eq:main_pb}, we will distinguish between two regimes. We prove the existence of a steady state in both cases, however the rates of convergence differ. In the first setting, only Hypotheses \ref{hypo:boundary} and \ref{hypo:k} are assumed. We prove that the rate of convergence is then bounded from above by the (optimal) polynomial rate of $(1+t)^{-d}$ derived for the free-transport equation in \cite{Bernou_2020, Bernou_2021, Bernou_2022}. In the second framework, $\sigma$ is almost everywhere bounded from below by a positive constant, and exponential convergence towards the steady state is derived. 
	\begin{hypo}
		\label{hypo:sigma_below}
		There exists $\sigma_0 > 0$ such that for almost all $x \in \Omega$, $\sigma(x) \ge \sigma_0$. 
	\end{hypo}
	
	
	Ultimately, our upper bounds rely on applications of Harris' theorems, in both the exponential case and the sub-exponential one. Accordingly, we obtain convergence results in the $L^1$ norm depending on some weighted $L^1$ norm of the initial data. 
	
	To define our weighted norms, we introduce the function
	\begin{align}
		 \label{eq:def_tau}
			\tau(x,v) := \left\{
			\begin{array}{lll} &\inf \{t > 0: x + tv \in \partial \Omega \} \quad &\hbox{ in } G \cup \Sigma_-, \\
			&0 \quad &\hbox{ in } \Sigma_+ \cup \Sigma_0, 
		\end{array} \right. 
	\end{align}
	where $\Sigma_0 = \{(x,v) \in \Sigma, v \cdot n_x = 0\}$. 
	The weights considered in this paper will take the following guise: for all $(x,v) \in \bar G$, the closure of $G$, for $d(\Omega)$ the diameter of $\Omega$ (see Subsection \ref{subsec:notations} below)
	\begin{align}
		\label{eq:def_m_alpha}
		m_\alpha(x,v) = \big(e^2 + \tfrac{d(\Omega)}{|v|c_4} - \tau(x,-v) + |v|^{2 \delta} \big)^\alpha
	\end{align}
	for $0 < \delta < \frac{\delta_k}{d} \wedge \frac{\varepsilon}{2d}$ that will be fixed from now on (see Hypothesis \ref{hypo:k} for the definition of $\delta_k$, \eqref{eq:moment_M} for the one of $\varepsilon$), for various $\alpha \in (0,d)$ and for $c_4 \in (0,1)$ a constant such that $(1-c_4)^4 = 1 - \beta_0$ (see Hypothesis \ref{hypo:boundary}). It is to be understood that any value $c_4 \in (0,1)$ can be considered for the case \CLBC \space and for the case \MBC \space when $\beta \equiv 1$. The reader may consider $c_4 = \frac12$ in the whole paper in those cases. Similarly, the choice $\varepsilon = d$ will be assumed throughout the paper in case \CLBC. 
	
	\begin{rmk}
		The form of the weights $m_\alpha$ may appear cumbersome at first sight. They are slight modifications of the natural weights of the form $(1 + \tau(x,v) + |v|^{2\delta})^\alpha$ used in \cite{Bernou_2021}: this change from $\tau(x,v)$ to $\frac{d(\Omega)}{|v|c_4} - \tau(x,-v)$ allows to also treat the Maxwell boundary condition in a unified framework.
	\end{rmk} 

 \noindent  Regarding the choice of $M$ in \MBC, our hypotheses encompass two natural cases:
    \begin{itemize}
        \item The Maxwellian at the wall: for all $(x,v) \in \Sigma_-$
\begin{align}
			\label{eq:def_M*}
			M(x,v) := \frac{1}{\theta(x) (2 \pi \theta(x))^{\tfrac{d-1}{2}}} e^{-\frac{|v|^2}{2 \theta(x)}}. 
		\end{align}
        This modeling of the reflection at the boundary dates back to Maxwell \cite{Maxwell_1879} and is ubiquitous in the mathematical literature -- see the discussion below. We shall refer to this condition as \MBCM.
        \item The $\kappa$-distributions at the wall. Those distributions with fat tails have been widely used in the physics literature since their introduction by Vasilyunas \cite{Vasyliunas_1968} in the study of low-energy electrons in the magnetosphere and, to model a reflection at the wall, take the form
        \begin{align}
        \label{eq:def_kappa}
            M(x,v) := c_M \frac1{\theta(x)^{\frac{d+1}{2}}} \big( 1 + \tfrac{|v|^2}{(\frac{d}2 + \varpi) \theta(x)} \big)^{-(d/2 + \varpi)}, \qquad (x,v) \in \Sigma_-,
        \end{align}
        where $\varpi > \frac12$ parametrizes the power-law of the distribution tail and $c_M > 0$ is a normalizing constant independent of $x$ and $v$ so that for all $x \in \partial \Omega$, $\int_{\Sigma^x_-} M(x,v) \, |v \cdot n_x| \, \d v= 1$. We shall refer to this choice as \MBCkappa. 
    \end{itemize}
    Formally, the $\kappa$-distribution \eqref{eq:def_kappa} approaches the Maxwellian one \eqref{eq:def_M*} as $\varpi \to \infty$. In this sense, one result of the present paper is that this whole class of reflection kernel yields similar asymptotic behavior of the collisional equation \eqref{eq:main_pb}. 

    \medskip 

    \noindent
    For all $f \in L^1(G)$, we use the notation
	\begin{align*}
		\langle f \rangle := \int_{G} f(x,v) \, \d v \, \d x. 
	\end{align*}
	We write $\|\cdot\|_{L^1}$ for the norm of $L^1(G)$, and for all $w : G \to [1, \infty)$, we set
	\begin{align*}
		L^1_w(G) := \Big\{f \in L^1(G), \int_G |f(x,v)| w(x,v) \, \d v \, \d x < \infty\Big\} \quad \hbox{ and } \quad \forall f \in L^1_w(G),  \|f\|_{w} := \|f w\|_{L^1}. 
		\end{align*}
	After proving that the problem \eqref{eq:main_pb} is well-posed under Hypotheses \ref{hypo:boundary} and \ref{hypo:k}, see Theorem \ref{thm:well_posed_trace}, we introduce the semigroup $(S_t)_{t \ge 0}$
	such that, for all $f \in L^1(G)$, for all $t>0$, $S_t f$ is the unique solution of \eqref{eq:main_pb} at time $t > 0$ belonging to $L^1(G)$. Our main results are written at this semigroup level. Throughout the paper, the constants $C, \kappa > 0$ are independent of time and initial data and are allowed to change from line to line. We sometimes write subscripts to emphasize dependencies, for instance $C_p$ if $C$ depends on some parameter $p$.  
	
	\begin{thm}
		\label{thm:main_1}
		Assume that Hypotheses \ref{hypo:boundary} and \ref{hypo:k} hold. Then for all $p \in (0,d)$, there exists a constant $C > 0$ such that for all $t \ge 0$, for all $f, g$ in $L^1_{m_p}(G)$ with $\langle f \rangle = \langle g \rangle$, there holds:
		\begin{align}
			\label{eq:polynomial_cvg_difference}
			\big\| S_t f - S_t g \big\|_{L^1} \le \frac{C}{(t+1)^p} \|f-g\|_{m_p}.
		\end{align}
		Under Hypotheses \ref{hypo:boundary}-\ref{hypo:sigma_below}, for all $q \in (0,d)$, there exist two constants $C, \kappa > 0$ such that for all $t \ge 0$, for all $f,g$ in $L^1_{m_q}(G)$ with $\langle f \rangle = \langle g \rangle$, there holds:
		\begin{align}
			\label{eq:exponential_cvg_difference}
			\big\| S_t f - S_t g \big\|_{L^1} \le C e^{- \kappa t} \|f-g\|_{m_q}. 
		\end{align}
	\end{thm}
	
	Three consequences can be drawn from this theorem, which form our main results.
	
	\begin{thm}
		\label{thm:main_2}
		Assume that Hypotheses \ref{hypo:boundary} and \ref{hypo:k} hold.
		\begin{enumerate}[i.]
			\item  There exists a unique $f_{\infty}$ such that for all $\epsilon \in (0, 1/2)$, $f_\infty \in L^1_{m_{d-\epsilon}}(G)$, $f_\infty \ge 0$, $\langle f_{\infty} \rangle = 1$ and
			\begin{align*}
				&v \cdot \nabla_x f_{\infty} = \mathcal{C}f_\infty, \qquad (x,v) \in G, \\
				&\gamma_- f_{\infty} = K \gamma_+ f_\infty, \qquad (x,v) \in \Sigma_-. 
			\end{align*}
			\item For all $p \in (0,d)$, there exists a constant $C > 0$ such that for all $t \ge 0$, for all $f \in L^1_{m_p}(G)$ with $f \ge 0$ and $\langle f \rangle = 1$, 
			\begin{align}
				\label{eq:cvg_polynomial}
				\big\| S_t(f - f_{\infty})\|_{L^1} \le \frac{C}{(1+t)^p} \|f - f_\infty\|_{m_p}.
			\end{align}
			\item If, additionally, Hypothesis \ref{hypo:sigma_below} holds, for all $q \in (0,d)$, there exists two constants $C, \kappa > 0$ such that for all $t \ge 0$, for all $f \in L^1_{m_q}(G)$ with $f \ge 0$ and $\langle f \rangle = 1$, 
			\begin{align}
				\label{eq:cvg_expo}
				\big\| S_t(f - f_{\infty}) \big\|_{L^1} \le C e^{- \kappa t} \|f - f_{\infty}\|_{m_q}. 
			\end{align}
		\end{enumerate}
		
	\end{thm}
	
	In Section \ref{sec:counter_ex}, we also assume that $f_\infty$ is uniformly bounded. In this setting, we present a counter-example showing that the exponential convergence can fail when Hypothesis \ref{hypo:sigma_below} does not hold. We also provide general exponential lower bound for the rate of convergence of \eqref{eq:main_pb}, as well as a polynomial lower bound for the case where $\sigma$ cancels on an open ball inside $\Omega$. 
	
	\begin{thm}
		\label{thm:lower_bounds}
		Assume Hypotheses \ref{hypo:boundary} and \ref{hypo:k}. 
		Let $f_\infty$ given by Theorem \ref{thm:main_2}, and suppose furthermore that $f_\infty \in L^{\infty}(G)$.
		Then,
		\begin{enumerate}
			\item for all $\alpha \in (0,d)$, the uniform decay rate $E(t)$ such that for all $f \in L^1_{m_\alpha}(G)$, with $f \ge 0$ and $\langle f \rangle = 1$, $t > 0$, 
			\begin{align*}
				\Big\| S_{t} f - f_{\infty}\|_{L^1} \le E(t) \|f-f_\infty\|_{m_\alpha} 
			\end{align*}
			satisfies $E(t) \ge C_{\alpha, \|f_\infty\|_{m_\alpha}} e^{-\sigma_\infty (1+\alpha) t}$ for all $t$ large enough;
			\item if $\sigma$ vanishes on an open ball of $\Omega$, for all $\alpha \in (0,d)$, there does not exists any constant $C, \kappa > 0$ such that for all $f \in L^1_{m_\alpha}(G)$ with $\langle f \rangle = 1$, for all $t > 0$,
			\begin{align}
				\label{eq:counter_ex}
				\Big\| S_{t} f - f_\infty \Big\|_{L^1} \le C e^{-\kappa t} \|f - f_\infty\|_{m_\alpha};
			\end{align}
			\item if $\sigma$ vanishes on an open ball of $\Omega$, for all $\alpha \in (0,d)$, the uniform decay rate $E(t)$ such that for all $f \in L^1_{m_\alpha}(G)$, with $f \ge 0$ and $\langle f \rangle = 1$, $t > 0$, 
			\begin{align*}
				\Big\| S_t f - f_{\infty}\|_{L^1} \le E(t) \|f-f_\infty\|_{m_\alpha} 
			\end{align*}
			satisfies $E(t) \ge C_{\alpha, \|f_\infty\|_{m_\alpha}} (t+1)^{-\alpha}$ for $t$ large enough.
		\end{enumerate}
	\end{thm}

	\medskip

	Table \ref{tab:summary_easy_frame} below summarizes our findings for this specific framework. 
		
		\begin{table}[h]
			\centering
			\captionsetup{width=.9\linewidth}
			\begin{tabular}{|p{4cm}|m{3cm}|m{3cm}|}
				\hline
				Hypothesis & Lower bound & Upper bound \\
				\hline 
					$\sigma \equiv 0$ on a ball $B \subset \Omega$
				 & $(t+1)^{-\alpha}$ & $(t+1)^{-\alpha}$ \\
				
				without Assumption \ref{hypo:sigma_below} & $e^{-\sigma_\infty(1+\alpha) t}$ & $ (t+1)^{-\alpha}$  \\
			 
				under Assumption \ref{hypo:sigma_below} & $e^{-\sigma_\infty(1+\alpha) t}$ &$e^{-\kappa t}$ \\
				\hline
			\end{tabular}
			
			\medskip 
			
			\caption{Convergence rate $E(t)$ from $L^1_{m_\alpha}(G)$ to $L^1(G)$, for $\alpha \in (0,d)$ and $f_\infty$ in $L^\infty(G)$. Bounds are given up to a constant independent of time and initial data.}
			\label{tab:summary_easy_frame}
		\end{table}
%
%
%
%
%
%
	
	Before turning to the motivations and to our review of the existing literature, we make a few remarks regarding those results. 
	
	\begin{rmk}[Use of $L^1$ weighted spaces]
		Aoki and Golse \cite[Proposition 3.1]{Aoki_2011} showed the non-existence of a uniform rate of convergence in $L^1(G)$ for general $L^1(G)$ initial data in the free-transport case with \MBCM, which is compatible with Hypotheses \ref{hypo:boundary} and \ref{hypo:k}. The uniform decay is indeed obtained here for initial data in some \textbf{weighted} $L^1$ spaces instead.
	\end{rmk} 
	
	\begin{rmk}[Constructive constants]
		The use of deterministic Harris' theorems to study the rate of convergence towards the steady state yields explicit constants \cite{Canizo_2023}. Those however depend on the constants appearing in the two conditions from which the proof is derived: the Lyapunov inequality and the Doeblin-Harris condition. While, in the former, constants are transparent, the ones from the latter depend here in a complicated fashion of the domain considered, see Remark \ref{rmk:nu_constructive}.
		Note also that we crucially use the stochastic nature of our boundary conditions: in case \CLBC, as $(\rp,\rt)$ converges to $(0,0)$ or as $(\rp, \rt) \to (0,2)$, i.e. as the reflection mechanism tends to the specular or the bounce-back boundary condition (see below Subsection \ref{subsubsec:boundary}), the constants from the Lyapunov conditions explode, see Proposition \ref{prop:lyapunov} and its proof. Similarly, in case \MBC, as $\beta_0 \to 0$, i.e. as the reflection mechanism tends to the specular one, our weights construction fails, since we require $c_4 \in (0,1)$ with $(1-c_4)^4 = (1-\beta_0)$. Regarding the lower bounds from Theorem \ref{thm:lower_bounds}, the constants appearing in front of the convergence rates are also constructive, but depend on the generally unknown values $\|f_\infty\|_{m_\alpha}, \|f_\infty\|_\infty$. 
		\end{rmk}
		
		\begin{rmk}[About the boundedness hypothesis in Theorem \ref{thm:lower_bounds}] It is known for several boundary conditions considered (see \cite{Esposito_2013} for case \MBCM, \cite{Chen_CL_2020} for case \CLBC) that, in the case of small temperature variations at the wall, the steady state of the full Boltzmann equation exists, is unique in the class of sufficiently regular functions, and is bounded. The boundedness hypothesis from Theorem \ref{thm:lower_bounds} thus appears natural at least in those frameworks.
		\end{rmk}
	
	\begin{rmk}[About the connectedness assumption]
	We assume that $\Omega$ is connected for simplicity. The case where $\Omega$ has finitely many connected components can also be dealt with, by splitting the densities and the corresponding steady states on each of those components. Further extensions seem really involved.
		\end{rmk} 
	
	\begin{rmk}[About a Doeblin condition]
		It might be possible to derive the exponential convergence from $L^1(G)$ to $L^1(G)$ under the additional Hypothesis \ref{hypo:sigma_below}, for instance by showing that, in this setting, the semigroup satisfies a Doeblin condition (rather than what we call a Doeblin-Harris one): for some $T > 0$ and a non-negative measure $\nu \not \equiv 0$, for all $(x,v) \in G$ and $f \in L^1(G; \R_+)$, 
		\[ S_T f(x,v) \ge \nu(x,v) \int_G f(y,w) \, \d y \, \d w, \]
		which is to be compared with the statement of Theorem \ref{thm:Doeblin-Harris} which only gives an upper bound to a restricted integral. Such a strategy was successful in \cite{Evans_2019} for the study of the degenerate linear Boltzmann equation in the torus. This could upgrade very slightly our results, since we only obtain exponential convergence from $L^1_{m_\epsilon}(G)$ to $L^1(G)$ for any $\epsilon >0$. We found however difficult to adapt the argument of \cite{Evans_2019} to a framework including our boundary conditions.
		\end{rmk}
	
	\begin{rmk}[Absence of perturbative arguments]
		We emphasize that our proofs do not rely on any perturbative arguments. We can thus treat the whole spectrum of accommodation coefficients for case \CLBC, that is $(\rp, \rt)$ in $(0,1] \times (0,2)$, and, for case \MBC, $\beta \in (\beta_0, 1]$ for any $\beta_0 > 0$ fixed. Similarly, we only assume continuity and positivity of the temperature, without requiring small variations around a constant.  
		\end{rmk}

    \begin{rmk}[All $\kappa$-distributions with $\varpi > \frac12$ yield the same behavior]
    As observed through the probabilistic approach of~\cite{Bernou_2022}, the crucial feature necessary to derive the polynomial convergence with rate $(t+1)^{d-}$ of Theorems \ref{thm:main_1} and \ref{thm:main_2} in case \MBC \, with $\beta \equiv 1$ is that $M$ should be bounded. Here, we recover this necessary hypothesis together with the moment condition \eqref{eq:moment_M}, and extend the result by saying that boundedness of $M$ and of a moment of order $1+\varepsilon$ is enough to recover exponential convergence, assuming the latter is provided by the collisional kernel (see the proof of Proposition \ref{prop:lyapunov}). The shared boundedness and moment properties explain that all $\kappa$-distributions of the form \eqref{eq:def_kappa} with $\varpi > \frac12$ (including the limiting one \eqref{eq:def_M*}) yield similar behaviors.
    \end{rmk}

	\subsection{Context, previous results and motivations}
	
	The linear Boltzmann equation is fundamental in kinetic theory and statistical physics. It describes the behavior of a dilute gas of particles encountering collisions with some background \cite{Cercignani_1988, Dautray_1990, Dautray_1993}.  
	Applications of this model span a wide range of disciplines: in physics, it is used to investigate neutron transport \cite{Cercignani_2000}, quantum scattering \cite{Erdos_2000}  and semiconductor device modeling \cite{Markowich_1990}. The linear Boltzmann equation has been derived in several contexts, see \cite{Breteaux_2014} for the case of a particle interacting with a random field, \cite{Bodineau_2015} for a study of hard-spheres, representing gas molecules. 
	The degenerate linear Boltzmann equation is a generalized version, adapted for instance to the study of radiative transfer systems inside which different parts of the space may have different transparencies. Our model set inside a bounded domain with stochastic boundary conditions is also reminiscent of the one presented in \cite{Bal_2002} for the study of photon migration within the skull, with applications in imagining of tumors and cerebral oxygenation \cite{Arridge_1999, Arridge_1997}.  
	
	In the past few years, the study of the linear Boltzmann equation, and of the BGK model \cite{BGK_1954} (also called the BKW model \cite{Welander1954} in the physics literature) where $k(x,v,v') = M_1(v')$ with
	\begin{align}
		\label{eq:def_M_1}
		M_1(v) = \frac{e^{-\frac{|v|^2}{2}}}{(2\pi)^{\frac{d}{2}}}, \qquad v \in \R^d,
	\end{align} combined with some boundary conditions have drawn a lot of interest within the mathematical community. There are two main reasons for this:
	\begin{itemize}
		\item those models have some physical relevance, with several well-identified applications;
		\item they present strong mathematical challenges, due to the delicate interaction between the transport operator with boundary conditions and the collision operator. 
 	\end{itemize}
 	We develop those two aspects in the next paragraphs. 
 	
 	\subsubsection{Physical features and boundary conditions}
 	\label{subsubsec:boundary}
 	
 	We present some key facts, and refer to \cite{Bernou_2021, Cercignani_2000} for more details. 
 	
 	\medskip 
 	
 	When modeling a gas inside a bounded domain $\Omega$, several choices of boundary conditions at $\partial \Omega$ are at disposal. The most simple ones are: 
 	\begin{enumerate}
 		\item  the bounce-back boundary condition : for all $(t,x,v) \in \R_+ \times \Sigma_-$,
 		\begin{align}
 			\label{eq:bounce_back}
 			 f(t,x,v) = f(t,x,-v); \end{align}
 		\item the specular reflection: for all $(t,x,v) \in \R_+ \times \Sigma_-$,
 		\begin{align}
 			\label{eq:specular}
 			f(t,x,v) = f\big( t,x, \eta_x(v) \big). 
 		\end{align} 
 	\end{enumerate}
 	Those conditions are unable to render the stress exerted by the
 	gas on the wall, and for this reason, Maxwell~\cite[Appendix]{Maxwell_1879}  introduced the pure diffuse reflection: for all $(t,x,v) \in \R_+ \times \Sigma_-$, taking the temperature $\theta \equiv 1$ independent of $x$, 
 	\begin{align}
 		\label{eq:diffuse}
 		f(t,x,v) = \frac{1}{(2 \pi)^{\tfrac{d-1}{2}}} e^{-\frac{|v|^2}{2}} \int_{\Sigma_+^x} f(t,x,w) |w \cdot n_x| \, \d w. 
 	\end{align} 
 	As opposed to \eqref{eq:bounce_back} and \eqref{eq:specular}, there is no correlation between the incoming velocities of particles hitting the wall and their outgoing ones in \eqref{eq:diffuse}. A first possible correction is to consider instead the Maxwell boundary condition \MBCM, a convex combination between the pure diffuse reflection and the specular one.
 	
 	Introduced at the beginning of the 1970's by Cercignani and Lampis \cite{Cercignani_Lampis_1971}, condition \CLBC \space provides a more delicate way to modify \eqref{eq:diffuse} to obtain those correlations. Its superior accuracy over the aforementioned models was exhibited numerous times,  both from numerical computations performed in the 1980's, and from physical experiments \cite{Alexandrychev_CL_1986, Markelov_CL_1982, Pantazis_CL_2011, Sharipov_CL_2002}, see in particular the recent work of Yamaguchi et al.~\cite{Yamaguchi_2016}. This paper was followed by a theoretical derivation of the coefficients in the context of hard spheres from Nguyen et al. \cite{Nguyen_2020}, who showed that the accommodation coefficients are independent of the shape of the domain, depend on the gas species considered,  and can, for some of those, be very different from the values $(1,1)$ corresponding to \eqref{eq:diffuse}. For instance, in a setting controlling temperature variations and pressure, an estimation for He was given in \cite[Table II]{Nguyen_2020}, with values $(\rp, \rt) = (0.15, 0.8)$. It is thus important to obtain mathematical results for the whole spectrum $(0,1] \times (0,2)$ of accommodation coefficients.

   The boundary condition \MBCkappa \space generalizes the one used in the study of fractional diffusion by Cesbron-Mellet-Puel \cite{Cesbron_Mellet_Puel_2019} and interpolates between their model and the specular reflection considered by Cesbron \cite{Cesbron_2018}. The collision kernel used in those papers also fits our framework, and is relevant to the study of plasma~\cite{Summers_Thorne_1991}. Although closely linked, the derivation of fractional/anomalous diffusion limit is a more general problem than the asymptotic convergence issue investigated here: for this reason, we postpone our discussion of those references to Section \ref{subsec:heavy_tail}. 
 	
 	\subsubsection{Mathematical motivations and previous results}
 	
	Equation \eqref{eq:main_pb} combines a first-order transport dynamics with two subtle relaxation effects in the velocity variable:
	\begin{itemize}
		\item the degenerate collision mechanism;
		\item a stochastic boundary operator. 
	\end{itemize} 
	Several results are already known regarding the long-time behavior of this kind of model. 
	
	\medskip 
	
	Consider first the sole transport dynamics with boundary conditions. Aoki and Golse \cite{Aoki_2011} where the first to question whether the thermalisation effect at the wall alone was enough to produce a spectral gap. For the diffuse reflection \eqref{eq:diffuse}, they identify the lack of uniform convergence for $L^1(G)$ initial data, and proved a convergence rate of $(1+t)^{-1}$, from some weighted $L^1$ space to $L^1$. This result was improved up to the optimal rate $\frac{1}{(1+t)^{d-}}$ in several subsequent articles by Kuo, Liu and Tsai \cite{Kuo_2013, Kuo_2014} and Kuo \cite{Kuo_2015} in a radial domain, and by Bernou-Fournier \cite{Bernou_2022} and Bernou \cite{Bernou_2020} in a $C^2$ bounded one, and ultimately culminated in the treatment of the more general Cercignani-Lampis boundary condition \cite{Bernou_2021}, for which the same rate of convergence was obtained. The key outcome of those research is that stochastic boundary conditions (\MBC \space  and \CLBC) provide only a polynomial rate of convergence in the $L^1(G)$ distance: there is no spectral gap for those dynamics. 
	
	\medskip 
	
	Next, we turn to hypocoercive equations, that is, dynamics combining a relaxation in the velocity variable with a transport operator. Those have been heavily studied in the past twenty years, and we will restrain to the equations closest to our framework. 
	The BGK model was studied in the torus by Mouhot and Neumann \cite{Mouhot_2006} who proved the existence of a spectral gap in $H^1$ norm. This toroidal case was also investigated, along with the case of the whole space with a confinement potential, by Dolbeault-Mouhot-Schmeiser \cite{Dolbeault_2009, Dolbeault_2015}, who gave a beautiful, simple proof of exponential decay using $L^2$ hypocoercivity which applies to a whole range of linear operators, including the linear Boltzmann equation. Those articles are part of the growing literature regarding hypocoercivity, which in some sense started from the work of Desvillettes, Hérau, Nier, Mouhot and Villani \cite{Desvillettes_2001, Desvillettes_2005, Herau_2005, Herau_2004, Mouhot_Villani_Review_2013, Villani_2009} among others, and benefited from earlier approaches, in particular the high-order Sobolev energy method of Guo \cite{Guo_2002}. At last, the degenerate linear Boltzmann equation investigated in this paper was studied in great details in the toroidal setting. In the case where $v \in V$ with $V$ bounded from below and above, Bernard and Salvarini \cite{Bernard_2013b} obtained exponential convergence towards the equilibrium under a geometric control condition. They also built in \cite{Bernard_2013} a counter-example showing that exponential convergence is not true in general. Later, Han-Kwan and Léautaud \cite{Han_Kwan_2015} used tools from control theory to deal with the case $v \in \R^d$ with a confinement potential, and obtained conditions about the spatial behavior of $k$ under which exponential convergence to the steady state occurs. They also characterized the latter under some extra hypotheses on $k$. In some sense, our paper extends the results of \cite{Bernard_2013b} to the case where $x \in \Omega$, $v \in \R^d$ with (stochastic) boundary conditions.
	
	\medskip  
	
	 This paper uses deterministic strategies inspired from probabilistic methods. Those tools, namely Doeblin and Harris theorems, were already used by Ca\~nizo-Cao-Evans-Yoldas \cite{Canizo_2020} to derive convergence rates towards equilibrium for the relaxation operator and the linear Boltzmann equation, in the torus and in the whole space with a confinement potential, some forms of the latter leading to polynomial rates of convergence, rather than exponential ones. Evans and Moyano \cite{Evans_2019} also recently used Doeblin's theorem to derive quantitative exponential convergence of the degenerate linear Boltzmann equation in the torus.

	\medskip  
	
	To conclude this literature review, we focus on models involving both a hypocoercive structure for the equation and non-deterministic boundary conditions. For deterministic boundary conditions (specular or bounce-back), we simply quote, among others \cite{Duan_Landau_2020, Guo_2010, Kim_Specular_Convex_2017,  Kim_Specular_2018}. 
	For the diffuse reflection with constant temperature, Guo \cite{Guo_2010} obtained exponential convergence in some weighted $L^{\infty}_{x,v}$ space for the linearized Boltzmann equation when $\Omega$ is smooth and convex, using his famous $L^2-L^{\infty}$ approach. Briant \cite{Briant_2017} extended this result to more general weights. Guo-Briant \cite{Briant_Maxwell_2016} upgraded those findings to get explicit constants and handle \MBCM. Regarding those topics (including the treatment of the linearized Landau equation), we also mention the recent \cite{Bernou_2022b} based on an enriched version of the Dolbeault-Mouhot-Schmeiser approach to hypocoercivity \cite{Dolbeault_2009, Dolbeault_2015}. When the reflection at the wall is diffuse with small temperature variations, Esposito-Guo-Kim-Marra \cite{Esposito_2013} showed the existence of a steady state and gave an exponential result of convergence in $L^\infty$ norm, see also \cite{Duan_2019}, but virtually nothing is none outside this case. The study of condition \CLBC \space in those collisional contexts has been very sparse, with the notable exception of the work of Chen \cite{Chen_CL_2020}, who extended the results from \cite{Esposito_2013}, under again an assumption of small temperature variations and strong hypotheses on the accommodation coefficients, that must be close to $(1,1)$. 
	The recent article of Dietert et al. \cite{Dietert_2022} is the closest to our framework, as it considers degenerate linear equations (namely, the linear Boltzmann equation and the linear Fokker-Planck equation) with confinement mechanisms that include the case of the diffuse reflection at constant temperature \eqref{eq:diffuse}. Using trajectorial methods and tools from control theory, the authors give conditions under which exponential convergence towards the equilibrium is achieved, in some $L^2$ norm, with constructive rates. The approach allows to treat several models and confinement mechanisms in a unified way. 
    \dd 
	
	\subsection{Contributions}
	
	The $L^2$-hypocoercivity tools mentioned above require the knowledge of the equilibrium and some form of separation of variables for it, as the velocity distribution is used as a weight to tailor appropriate functional spaces. So far those methods have given limited insights about the asymptotic behavior of the solutions when the temperature varies at the boundary. This framework is however meaningful from a physical point of view in our context: considering degenerate models implies that the thermalization effects are different in various regions of space. Extending these features up to the boundary, which amounts to considering wall temperatures that also change with the position, is thus a very natural assumption. The linear Boltzmann equation with variable temperature at the boundary is also an interesting framework for the study of the Fourier law in the kinetic regime \cite{Esposito_2013}.
	
	\medskip

	In this paper, our main contributions apply to several stochastic boundary conditions (\CLBC, \MBCM \space and \MBCkappa) with no assumptions on the temperature besides boundedness and continuity, and in a general $C^2$ bounded domain. We obtain five main results: 
	\begin{enumerate}
		\item the existence and uniqueness of the steady state for the (degenerate) linear Boltzmann equation;
		\item an exponential rate of convergence towards this steady state, in $L^1(G)$, for the linear Boltzmann equation (i.e. $\sigma \equiv 1$);
		\item an exponential rate of convergence towards the steady state, in $L^1(G)$, for the degenerate linear Boltzmann equation under an additional control condition;
		\item a polynomial rate of convergence towards the steady state, in $L^1(G)$, for the degenerate linear Boltzmann equation without the additional control condition;
		\item a precise picture of the convergence, including lower bounds on the rates, under an additional boundedness assumption of the steady state that is known to hold for the full Boltzmann equation in the case of small temperature variations.
	\end{enumerate}
	
    In addition to those, we present in Section \ref{sec:applications} two further applications. 
    In Section \ref{subsec:applications_variyng} we present a linear relaxation model that can be seen as the counterpart to the degenerate linear BGK one in the case of varying temperature at the boundary. This model is very natural in the study of multi-species interaction combined with boundary effects, and we provide quantitative estimates of convergence in Corollary \ref{coroll:extended_BGK}. 
    In Section \ref{subsec:heavy_tail} the collision kernel is of the form $k(x,v,v') = \sigma(x) \mu_\varpi(v')$ for all $(x,v,v')$ in $ G \times \R^d$, where $\mu_\varpi$  behaves as $|v|^{-d - 2\varpi}$ as $|v| \to \infty$ -- that is, like a $\kappa$-distribution --  with $\varpi > 0$. Combining such choice with the corresponding boundary condition \MBCkappa \space is particularly relevant to the study of fractional/anomalous diffusion limits, for which our result can be read as a first step towards a generalization of the recent findings of \cite{Cesbron_2018, Cesbron_Mellet_Puel_2019}. 

	\medskip 
	
	In our opinion, our results regarding the convergence rates should be read as follows:
	\begin{itemize}
		\item Under the sole Hypotheses \ref{hypo:boundary} and \ref{hypo:k}, the decay is due to the interplay between the free-transport dynamics and the boundary condition, hence the rate of convergence can not beat the one obtained for the free-transport problem in \cite{Bernou_2020, Bernou_2021}. The stochastic nature of the boundary condition is key to the mixing in both space and velocity: in terms of trajectories this is best understood at the level of the Doeblin-Harris condition, Theorem \ref{thm:Doeblin-Harris}, where it is shown, roughly, that particles with controlled velocity and next boundary collision time span the whole phase space.
		\item Under the additional Hypothesis \ref{hypo:sigma_below}, the collision operator is sufficiently involved into the dynamics to provide further mixing, and therefore additional decay, eventually leading to some exponential convergence.
        \item In case \MBC \space, the two key hypotheses allowing the derivation of our results are the boundedness of $M$ and the existence of a $1^+$ moment \eqref{eq:moment_M}. As those properties are common to all $\kappa$-distributions \eqref{eq:def_kappa} with $2\varpi > 1$, including the limiting one \eqref{eq:def_M*}, such heavy-tailed reflection kernels share the same asymptotic behavior of the solution to \eqref{eq:main_pb} as the wall Maxwellian \eqref{eq:def_M*}.
	\end{itemize}

 	\subsection{Strategy and plan of the paper} \
 	
 	\medskip
 	
  Section \ref{sec:applications} presents four applications with given choices of $k$. We start with the linear BGK equation and a general linear Boltzmann model. The third application is the case of a heavy-tailed collisional kernel described above, and discussed in more details in Section \ref{subsec:heavy_tail}, also in connection with the question of deriving diffusion limits. Of particular interest is the setting of Section \ref{subsec:applications_variyng}, namely the interaction between two gas species, which corresponds to $k(x,v,v')= \tilde f_\infty(x,v')$, $(x,v,v') \in G \times \R^d$, for $\tilde f_{\infty}$ the steady state of the full Boltzmann equation set inside the domain with variable boundary temperature. This provides a relaxation model which is more relevant than the usual linear BGK one for the case where the temperature varies both inside the domain and at the boundary.
  
  \medskip 
 	
 	Ultimately, this paper relies on a deterministic Doeblin-Harris type argument, in the spirit of \cite{Bernou_2020, Bernou_2021}, see also \cite{Canizo_2023} and the recent review \cite{Yoldas_2023}. The core of our strategy builds upon the structure (and known results) of the underlying free-transport operator. Under Hypotheses \ref{hypo:boundary} and \ref{hypo:k}, Problem \eqref{eq:main_pb} is a bounded perturbation of the two models studied in \cite{Bernou_2020} and \cite{Bernou_2021}. This is the key argument providing our well-posedness result and important features of the trace, in Section \ref{sec:well-posed}. 
 	
 	\medskip 
 	
 	Section \ref{sec:Lyapunov} is devoted to our derivation of the Lyapunov conditions. The main point is as follows: when differentiating $\|S_t f\|_{m_\alpha}$ for some $f \in L^1_{m_\alpha}(G)$, $t \ge 0$ and $\alpha \in (1,d)$, one obtains
 	\begin{align}
 		\label{eq:lyapunov_q}
 		\frac{d}{dt} \|S_t f\|_{m_\alpha} \le Q_\alpha(S_t f) - \int_G v \cdot \nabla_x\big(m_\alpha\big) |S_t f| \, \d v \, \d x - \int_G \sigma(x) |S_t f| \, m_\alpha \, \d v \, \d x, 
 	\end{align} 	
 	where $Q_\alpha(S_t f)$ represents the sum of the boundary and gain terms. 
 	One shows the equality $-v \cdot \nabla_x m_\alpha = - \alpha \, m_{\alpha - 1}$ on $G$ by construction of the weights. Under Hypotheses \ref{hypo:boundary} and \ref{hypo:k} we just ignore the last term on the right-hand-side (r.h.s.) of \eqref{eq:lyapunov_q} and the decay of the norm is given by the term $-\alpha \|S_t f\|_{m_{\alpha-1}}$ originating from the free-transport dynamics rather than from the collision operator. Under the additional Hypothesis \ref{hypo:sigma_below}, we ignore this term in $m_{\alpha - 1}$ and get the decay from the loss term of the collision operator, in the form of $- \sigma_0 \|S_t f\|_{m_\alpha}$. The treatment of the boundary terms in $Q_\alpha$ within \eqref{eq:lyapunov_q} is a delicate point, for which we adapt the previous strategies introduced in \cite{Bernou_2020, Bernou_2021}. We integrate \eqref{eq:lyapunov_q} on $[0,T]$, $T > 0$, as one can show that integrated boundary fluxes of $f$ are controlled by $C(1+T)\|f\|_{L^1}$ for both boundary conditions. 
	We conclude that for all $\alpha \in (1,d)$, there exist $K_1, K_2 > 0$ two constants such that:
 	 \begin{itemize}
 	 	\item under Hypotheses \ref{hypo:boundary} and \ref{hypo:k}, for all $T > 0$, for all $f \in L^1_{m_\alpha}(G)$,
 	 		\begin{align}
 	 			\label{eq:Lyapu_poly_Intro}
 	 		\|S_T f\|_{m_\alpha} + \alpha \int_0^T  \|S_s f\|_{m_{\alpha-1}} \, \d s \le \|f\|_{m_\alpha} +  K(1+T) \|f\|_{L^1}.
 	 	\end{align}
 	 	\item If Hypothesis \ref{hypo:sigma_below} also holds, for all $T > 0$, all $f \in L^1_{m_{\alpha}}(G)$,
 	 	\begin{align}
 	 		\label{eq:Lyapu_expo_Intro}
 	 		\|S_T f\|_{m_{\alpha}} + \sigma_0 \int_0^T  \|S_s f\|_{m_{\alpha}} \, \d s \le \|f\|_{m_{\alpha}} +  K_2(1+T) \|f\|_{L^1}.
 	 	\end{align}
 	 \end{itemize} 
 	 
 	 \dd 
 	
 	\medskip

 	Section \ref{sec:Doeblin} focuses first on the Doeblin-Harris condition in Subsection \ref{subsec:Doeblin}. There, the Duhamel formula \eqref{eq:Duhamel} renders very concretely our use of the free-transport dynamics. Indeed, we first show that, for all $(t,x,v) \in \R_+ \times G$,
 	\begin{align*}
 		S_t f(x,v) \ge \mathbf{1}_{\{\tau(x,-v) \le t\}} e^{-\int_0^{\tau(x,-v)} \sigma(x-sv) \, \d s} S_{t - \tau(x,-v)} f \big(x-\tau(x,-v)v, v\big),
 	\end{align*}
 	 which allows us to ignore the gain collision mechanism (we only need some boundedness of $\sigma$) to derive the minoration condition: for all $\Lambda$ large enough, there exist $T(\Lambda) > 0$ and a non-negative measure $\nu \not \equiv 0$ on $G$ such that for all $(x,v) \in G$, for all $f_0 \in L^1(G)$, $f_0 \ge 0$, 
 	 \begin{align}
 	 	\label{eq:DH_intro}
 	 	S_{T(\Lambda)} f_0(x,v) \ge \nu(x,v) \int_{\{(y,w) \in G:\,  m_1(y,w) \le \Lambda \}} f_0(y,w) \, \d y \, \d w.
 	 \end{align} 
 	
 	\medskip 
 	
 	\dd 
 	Once conditions \eqref{eq:Lyapu_poly_Intro}-\eqref{eq:Lyapu_expo_Intro} and \eqref{eq:DH_intro} are established, we follow in Subsection \ref{subsec:proof_thm_1} a strategy reminiscent of the one in \cite{Bernou_2020, Bernou_2021, Canizo_2023}. Roughly, the core mechanism is as follows: inside the sublevel sets of the weight functions, \eqref{eq:DH_intro} provides some contraction. Outside of those sublevel sets, the Lyapunov conditions \eqref{eq:Lyapu_poly_Intro}-\eqref{eq:Lyapu_expo_Intro} tell us how fast the dynamics return to them. The speed of convergence can thus be read at this level. This strategy is in some sense analogous to a probabilistic coupling -- one such for the free-transport dynamics is performed in \cite{Bernou_2022} -- but relying on the framework introduced by Hairer-Mattingly \cite{Hairer_2016, Hairer_2011} and refined by Ca\~nizo-Mischler \cite{Canizo_2023} allows to escape the corresponding cumbersome construction -- especially in models like the ones investigated here, whose probabilistic writing would involve several sources of randomness --  by playing with weighted norms instead.  
	 We obtain the existence and uniqueness of the steady state, and some rate of convergence towards it. This is one of the main strengths of Doeblin-Harris type arguments, particularly with respect to hypocoercivity methods, which makes them well-tailored for the study of models whose parameterization is more involved: the knowledge of the steady state is not required \textit{a priori}. 
	 
	 \dd 
 	
 	\medskip

 	Section \ref{sec:counter_ex} is devoted to the proof of Theorem \ref{thm:lower_bounds}. By building an appropriate initial data $f_\epsilon$ depending on some parameter $\epsilon \in (0,1)$, and by using a comparison principle with the solution of the  problem
 \begin{align*}
 	\left\{
 	\begin{array}{lll}
 		&\partial_t \Phi + v \cdot \nabla \Phi = - \sigma(x)\Phi \qquad &\mathrm{in } \quad \R_+ \times G, \\
 		&\gamma_- \Phi = 0 \qquad &\mathrm{on } \quad \R_+ \times \Sigma_-, \\
 		&\Phi_{| t = 0} = f_\epsilon, \qquad &\mathrm{in } \quad G, 
 	\end{array}
 	\right.
 \end{align*}
 whose solution is explicitly given by the method of characteristics, we derive a general inequality on the uniform convergence rate $E(t)$. We then draw conclusions by choosing appropriately $\epsilon$. This strategy is in part inspired from \cite{Aoki_2011}.

	\subsection{Notations}
	\label{subsec:notations}
	We write $\bar B$ for the closure of any set $B$. We denote by $C^1_c(E)$ and $C^\infty_c(E)$ the space of test functions, $C^1$ and $C^\infty_c$ with compact support, respectively, on $E$. We write $d \zeta(x)$ for the surface measure at $x \in \partial \Omega$. For a function $\phi$ on $(0,\infty) \times \bar{G}$, we denote $\gamma_{\pm} \phi$ its trace on $(0,\infty) \times \Sigma_{\pm}$, under the assumption that this object is well-defined. We write $W^{1,\infty}(\R^d; \R)$ for the space of functions $g$ admitting a weak derivative, $\nabla g$, such that both $g$ and $\nabla g$ belong to $L^\infty(\R^d)$. We write $\vertiii{H}_{A \to B}$ for the operator norm of $H$ acting between the two Banach spaces $A$ and $B$. 
	
	\noindent Throughout the paper, $0 < \delta < \frac{\delta_k}{d} \wedge \frac{\varepsilon}{2d}$ is fixed, with $\delta_k$ given by Hypothesis \ref{hypo:k}, $\varepsilon$ given by \eqref{eq:moment_M} in case \MBC \space and equal to say, $d$ in case \CLBC. The constant $\chi_M$ is given by \eqref{eq:property_M}. We denote $d(\Omega) = \sup_{x,y \in \bar \Omega} |x-y|$ the diameter of $\Omega$, which is finite by assumption. For $h \in \R$, we define $\lfloor h \rfloor := \inf\{z \le h, z \in \mathbb{Z}\}$. In the whole paper, the positive constants $C$ and $\kappa$ depend only on the parameters (and not on the time nor on the initial data), and are allowed to change from line to line. We write subscripts when we wish to emphasize some dependency, e.g. $C_\alpha$ is a constant depending on $\alpha$ which can vary from line to line. We write $\sigma_\infty$ for the upper bound of $\sigma$, which is well-defined under Hypothesis \ref{hypo:k}. For two random variables $X, Y$ defined on a probability space with the same distribution, we write $X \overset{\mathcal{L}}{=} Y$. 
	
	\noindent For all $(x,v) \in \partial \Omega \times \R^d$,
	\[ \eta_x(v) := v - 2 (v\cdot n_x) n_x. \]
	In particular $\eta_x$ maps $\Sigma^x_{\pm}$ to $\Sigma^x_{\mp}$ and $v \to \eta_x(v)$ has Jacobian $1$. We sometimes have to distinguish between both boundary conditions, in which case we write \MBC \space and \CLBC \space to refer to the two settings of Hypothesis \ref{hypo:boundary}. 
	
	\section{Applications}
	\label{sec:applications}
	
	We detail in this section several collision kernels fitting into our framework. We begin with models associated to the kinetic theory of gas, before discussing an application to heavy-tails kernel, mostly found in the study of plasma \cite{Summers_Thorne_1991}. From physical considerations, the most natural boundary conditions for the first three paragraphs below are \CLBC \, and \MBCM: we shall only discuss those there. 
    
	\subsection{Linear relaxation kernel against the Maxwellian}
	
	We set, for all $(x,v,v') \in G \times \R^d$, $k(x,v,v') = \sigma(x) M_1(v')$ with $\sigma \in L^\infty(\Omega; \R_+)$ and $M_1$ given by \eqref{eq:def_M_1}. This corresponds to the so-called degenerate linear BGK model, whose collision operator is given, for $f \in L^1(G)$, $(x,v) \in G$, by
	\begin{align*}
		\mathcal C f(x,v) = \sigma(x) \Big( M_1(v) \int_{\R^d} f(x,v') \, \d v' - f(x,v) \Big). 
	\end{align*}

    	\subsection{Linear Boltzmann equation} We set, for all $(x,v,v') \in G \times \R^d$, \[ k(x,v,v') = \sigma(x) p(v,v'), \] with $\int_{\R^d} p(v,v') \, \d v' = P$, $P > 0$ constant, $\sup_{v \in \R^d} \int_{\R^d} p(v,v') |v'|^{2 \delta_k} \, \d v' < 0$ for some $\delta_k \in (0, \frac12)$, and $\sigma \in L^\infty(\Omega; \R_+)$. This generalizes the previous model and includes the case $\sigma \equiv 1$, which corresponds to the (non-degenerate) linear Boltzmann equation, and the case where $p(v,\cdot) \subset V$ with $V \subset \R^d$ bounded from above and below, as considered by Bernard-Salvarani \cite{Bernard_2013b} on the torus.

	\subsection{Relaxation against the steady state of a Boltzmann equation.}
	\label{subsec:applications_variyng}
	We present here a model which captures specific features of the case when the temperature varies at the boundary. The linear Boltzmann equation is a classical model for the interaction between two species of gas, say (A) and (B), when one of the species is more dense than the other \cite{Cercignani_1988}. Consider the following setting: species (A) is more dense, and has already reached a steady state inside the domain $\Omega$. Species (B) is less dense, to the point where inner collisions between particles of species (B) can be neglected in its evolution. 
	In the case \MBC \space with $M$ given by \eqref{eq:def_M*} and $\beta \equiv 1$ and small temperature variations, it is known \cite[Theorem 1.1]{Esposito_2013}, see also \cite{Guiraud_1972, Guiraud_1975},  that species (A), whose dynamics can be described by a full Boltzmann equation, admits a steady state, which depends on both $x$ and $v$ and that we denote $f_{A,\infty}$. Furthermore, 
	\begin{enumerate}
		\item $f_{A,\infty} \in L^\infty(G)$, 
		\item $\sup_{x \in \Omega} \int_{\R^d} f_{A,\infty}(x,v) |v|^{2\delta_k} \, \d v \lesssim 1$ for all $\delta_k \in (0, 1/2)$.
	\end{enumerate}
	
	Upon imposing more precise moment conditions on $f_{A,\infty}$, its uniqueness is again known \cite{Esposito_2013}. In case \CLBC, analogous results are available, provided that the temperature variations are small and that the accommodation coefficients are close to the values $(1,1)$ \cite[Corollary 2]{Chen_CL_2020}.
	
	In the study of space-dependent thermal exchanges, it is thus natural to study the model \eqref{eq:main_pb} for the density $f_{B}$ of the species (B) with the choice 
	\[ k(x,v,v') = f_{A,\infty}(x,v'), \qquad (x,v,v') \in G \times \R^d \] 
	(note that of course, $\int_{\R^d} k(x,v,v') \, \d v'= \sigma(x)$ is independent of $v$ in this framework) that is, to study the following evolution problem for $f_B$:
	\begin{align}
		\label{eq:pb_fB}
		\left\{
		\begin{array}{lll} 
			&\partial_t f_{B}(t,x,v) + v \cdot \nabla_x f_B(t,x,v) = f_{A,\infty}(x,v) \int_{\R^d}  f_B(t,x,v') \, \d v' \\
			&\qquad \qquad \qquad \qquad - f_B(t,x,v) \int_{\R^d} f_{A,\infty}(x,v') \, \d v', \qquad &(t,x,v) \in \R_+ \times G, \\
			&\gamma_- f_B = K \gamma_+ f_B, \qquad &(t,x,v) \in \R_+ \times \Sigma, \\
			&f_B(0,x,v) = f_0(x,v), \qquad &(x,v) \in G, 
		\end{array} 
		\right.
	\end{align} 
	for some initial data $f_0 \in L^1_{m_\alpha}(G)$, $\alpha \in (0,d)$. 
	It is clear from the conditions detailed above and satisfied by $f_{A,\infty}$ that this choice of $k$ satisfies Hypothesis \ref{hypo:k}.  
	
	\medskip 
	
	 Our results directly lead to the following corollary. 
	 
	 \begin{coroll}
	 	\label{coroll:extended_BGK}
	 	Under Hypothesis \ref{hypo:boundary}, the problem \eqref{eq:pb_fB} is well-posed. We write $(S_{B,t})_{t \ge 0}$ for the associated $\mathcal{C}_0$-stochastic semigroup given by Theorem \ref{thm:well_posed_trace}. 
	 	
	 	\begin{enumerate}
	 	
	 	\item There exists a unique steady state for the problem \eqref{eq:pb_fB}, $f_{\infty,B}$ such that for all $\alpha \in (0,d)$, $f_{\infty,B}$ in $L^1_{m_\alpha}(G)$, $f_{\infty,B} \ge 0$ and $\langle f_{\infty,B} \rangle = 1$. 
	 	
	 	\item For all $p \in (0,d)$, there exists a constant $C > 0$ such that for all $t \ge 0$, for all $f \in L^1_{m_p}(G)$ with $f \ge 0$ and $\langle f \rangle = 1$, 
	 	\begin{align*}
	 		\big\| S_{B,t}(f-f_{\infty,B}) \big\|_{L^1} \le \frac{C}{(1+t)^p} \big\|f - f_{\infty,B}\big\|_{m_p}. 
	 	\end{align*}
	 	
	 	\item Assume that $f_{A,\infty}$ is continuous in $G$, with, for all $x \in \Omega$, $\int_{\R^d} f_{A,\infty}(x,v) \, \d v > 0$. Then for all $q \in (0,d)$, there exist two constants $C, \kappa > 0$ such that for all $t \ge 0$, all $f \in L^1_{m_q}(G)$ with $f \ge 0$ and $\langle f \rangle = 1$, 
	 	\begin{align*}
	 		\big\| S_{B,t}(f-f_{\infty,B}) \big\|_{L^1} \le  C e^{-\kappa_t} \big\|f - f_{\infty,B}\big\|_{m_q}. 
	 	\end{align*}
	 	\end{enumerate}
	 \end{coroll} 
	 
	 \begin{proof}
	 	The well-posedness and the existence of the associated $C_0$-stochastic semigroup are given by Theorem \ref{thm:well_posed_trace}. Points \textit{(1)} and \textit{(2)} are given by Theorem \ref{thm:main_2}. Point \textit{(3)} follows from the fact that, by compactness, those hypotheses on $f_{A,\infty}$ imply $\inf_{x \in \Omega} \sigma(x) > 0$, so that Hypothesis \ref{hypo:sigma_below} is satisfied, and Point \textit{iii.} of Theorem \ref{thm:main_2} applies.  
	 \end{proof}
	 
	 It is worth noting that in case \MBCM \space with $\beta \equiv 1$, when $\Omega$ is convex, it is known that $f_{A,\infty}$ is continuous in $G$. A further refinement is provided in \cite[Theorem 1.2]{Esposito_2013} and shows that, at least in the case where the temperature variations are very small (that is $\delta \ll 1$ in the notations of \cite{Esposito_2013}), one should expect $\int_{\R^d} f_{A,\infty}(x,v) \, \d v > 0$ for all $x \in \Omega$. Point \textit{(3)} of Corollary \ref{coroll:extended_BGK} is thus relevant in this situation. 

     \subsection{Linear relaxation kernel with heavy tails}
\label{subsec:heavy_tail}
	
	For all $(x,v,v')$ in  $G \times \R^d$, we consider a collision kernel of the form $k(x,v,v') = \sigma(x) \tilde \mu_\varpi(v')$ with $\sigma \in L^\infty(\Omega; \R_+)$ and 
	\[ \mu_\varpi(v) = Z_{d,\varpi}^{-1} \mathcal{F}^{-1}\big(e^{-\frac{|\xi|^{2\varpi}}{2\varpi}} \big) \]
	where $Z_{d,\varpi} > 0$ is a constant so that $\int_{\R^d} \mu_{\varpi}(v) \mathrm{d} v = 1$ and where $\mathcal{F}^{-1}$ denotes the inverse Fourier transform. Away from the origin, the Fourier transform of $\mu_\varpi$ is smooth and rapidly decaying at infinity. In fact, the following more precise bounds are available, see \cite[Theorem 3.1]{Bogdan_2003}: for some constant $c_1 > 0$, for all $v \in \R^d$,
	\begin{align*}
	\frac{c_1^{-1}}{|v|^{d + 2\varpi} + 1} \le \mu_\varpi(v) \le \frac{c_1}{|v|^{d + 2\varpi} + 1}. 
	\end{align*}
    In this sense, this collision kernel can be seen as a generalized $\kappa$-distribution \eqref{eq:def_kappa}. 
	
	 The collision operator is given, for $f \in L^1(G)$, $(x,v) \in G$, by
	\begin{align}
	\label{eq:collision_fat_tail}
	\mathcal C f(x,v) = \sigma(x) \Big( \mu_\varpi(v) \int_{\R^d} f(x,v') \, \d v' - f(x,v) \Big). 
	\end{align}
	
	Clearly, this kernel fits Assumption \ref{hypo:k} with $\delta_k < \varpi$, as for such choice 
	\[ \int_{\R^d} \frac{|v|^{2 \delta_k}}{|v|^{d + 2\varpi} + 1} \, \mathrm{d}v < \infty, \]
	and our result thus apply. Notice that, while we are able to treat the whole physical range $\varpi > 0$ for the collision kernel, the reflection kernel is limited to the (still physical) range $\varpi > \frac12$. 
    
	Of particular interest is the case $\sigma \equiv 1$ of the scattering model with heavy-tailed collision kernel. Under a rescaling of the collision operator \eqref{eq:collision_fat_tail} with scale parameter (mean-free path) $\epsilon$, the (diffusion) limit equation as $\epsilon \to 0$ involves a fractional diffusion operator. While deriving such limits in a domain is an intricate question in general which goes beyond the scope of this paper, we notice that the scattering operator \eqref{eq:collision_fat_tail} was studied previously in this perspective, in the whole space without confinement, see~\cite{Bouin_2022, Jara_2009, Mellet_2010} and in a domain by Cesbron with the specular \cite{Cesbron_2018} reflection at the boundary and by Cesbron-Mellet-Puel with the diffuse boundary condition \cite{Cesbron_Mellet_Puel_2019}. In this context, as noted in this last paper, the choice of boundary condition at the kinetic level directly influences the limiting fractional diffusion obtained. The derivation in \cite{Cesbron_Mellet_Puel_2019} starts from a kinetic model in which the reflection at the boundary is given by a constant-temperature version of a $\kappa$-distribution, for which a qualitative result of convergence towards thermodynamic equilibrium is obtained. Our result with \MBCkappa \space and $\beta \in (0,1)$ provides a quantitative convergence when the boundary condition is an interpolation between the situation of \cite{Cesbron_2018} and \cite{Cesbron_Mellet_Puel_2019}, and as such, a potential starting point for the investigation of the associated limit fractional diffusion model.

We note that a similar problem is the one of linking the discrete kinetic Fokker-Planck equation with its diffusion limit, see~\cite{Mischler_2017}. For the latter, analytic approaches based on Harris-type theorems (which are probabilistic in nature) have recently emerged, see e.g.~\cite{Canizo_2024} and provide a nice bridge between probabilistic approaches used for other models -- mainly based on couplings and properties of the Brownian motion, see~\cite{Fournier_2020, Fournier_2021, Jara_2009} and the references therein -- and PDE techniques (based on entropy or spectral analysis) found e.g. in~\cite{Bouin_2022, Mischler_2017}. 
	
	\section{Setting, well-posedness and trace theory}
	\label{sec:well-posed}
	
	\subsection{Associated semigroup}

	We gather our well-posedness result and some key elementary properties in the next theorem. Note that the boundary operator $K$ given by \eqref{eq:def_K} is non-negative and has norm $1$. In case \CLBC, it follows from the normalization property \eqref{eq:normalization_basic}. In case \MBC, it is easily obtained from \eqref{eq:property_M}: for all $(x,u) \in \Sigma_+$,
	\begin{align}
		\label{eq:normalization_2}
		\int_{\Sigma_-^x} R(u \to v; x) |v \cdot n_x| \, \d v &= \beta(x) \int_{\Sigma_-^x} M(x,v) |v \cdot n_x| \, \d v \nonumber  \\
		&\quad + (1-\beta(x)) \frac{|\eta_x(u) \cdot n_x|}{|u \cdot n_x|} = \beta(x) + (1-\beta(x)) = 1,
	\end{align}
	where we used $|u \cdot n_x| = |\eta_x(u) \cdot n_x|$. 
	
	\begin{thm}[Well-posedness, mass conservation, contraction property and trace equality]
		\label{thm:well_posed_trace}
Assume Hypotheses \ref{hypo:boundary} and \ref{hypo:k} hold. There exists a $C_0$-stochastic semigroup $(S_t)_{t \ge 0}$ associated to the problem \eqref{eq:main_pb} in $L^1(G)$. That is, for all $t \ge 0$,  $f \in L^1(G)$, $(S_t f)_{t \ge 0}$ is the unique solution in $L^{\infty}([0,\infty), L^1(G))$ of \eqref{eq:main_pb} with initial condition $f$. Moreover, for all $f \in L^1(G)$,
\begin{enumerate}[i.]
	\item for all $t \ge 0$, the trace of $S_t f$, denoted $\gamma_t f$, is well-defined, with $(\gamma_t f)_{t \ge 0}$ an element of the space $L^1_{\loc}([0, \infty) \times \Sigma, (v \cdot n_x)^2 \, \d v \, \d \zeta(x) \, \d t)$ such that the Green's formula is satisfied: 
	for all $t_0, t_1$ in $\R_+$ and $\varphi \in C^1_c(\R_+ \times \bar G)$ with $\varphi \equiv 0$ on $\R_+\times \Sigma_0$: 
	\begin{align*}
		&\int_{t_0}^{t_1} \int_G \big[ S_t f \big( \partial_t + v \cdot \nabla_x \big) \varphi  + \varphi \, \mathcal{C} f \big]  \, \d v \, \d x \, \d t \\
		&\qquad = \Big[ \int_G S_t f \,  \varphi \, \d v \, \d x \Big]_{t_0}^{t_1} + \int_{t_0}^{t_1} \int_{\Sigma} (\gamma_t f) (v \cdot n_x) \, \varphi \, \d v \, \d \zeta(x) \, \d t. 
	\end{align*}
	 We also have the renormalization property: for all $\beta \in W^{1,\infty}(\R)$, $t \ge 0$
	\[ \gamma_t \beta(f) = \beta (\gamma_t f).   \]
	\item  The mass is conserved: for all $t \ge 0$,
		\begin{align}
			\label{eq:mass_conservation}
			\int_{G} S_t f(x,v) \, \d x \, \d v = \int_G f(x,v) \, \d x \, \d v.
		\end{align}
	\item For all $t \ge 0$,
\begin{align}
	\label{eq:contraction_L_1}
	\|S_t f\|_{L^1} \le \|f\|_{L^1}.
\end{align}
\item  The semigroup $(S_t)_{t \ge 0}$ is non-negative. 
\end{enumerate}
	\end{thm}

\begin{proof}
	
	\medskip 
	
	\textbf{Step 1: well-posedness.}
	As the boundary operator is conservative and stochastic, one can show that the associated free-transport problem, corresponding to \eqref{eq:main_pb} with $\mathcal{C} \equiv 0$, is governed by a $C_0$-stochastic semigroup $(T_t)_{t \ge 0}$, i.e. a non-negative, mass-conservative semigroup such that, for $f_0 \in L^1(G)$, for all $t \ge 0$, $T_t f_0 = f(t,\cdot)$ is the unique solution in $L^{\infty}([0,\infty); L^1(G))$ to the free-transport problem taken at time $t$. In case \CLBC, this was obtained by Cercignani and Lampis \cite{Cercignani_Lampis_1971}, along with the fact that $(S_t)_{t \ge 0}$ is a contraction semigroup, see also \cite{Bernou_2021}. For case \MBCM \space a proof can be found in \cite{Bernou_2020} which applies directly to the whole case \MBC. 
	
	\medskip 
	
	Turning to \eqref{eq:main_pb}, note that the corresponding operator is nothing but a perturbation of the free-transport equation, with either boundary conditions, by the operator $\mathcal{C}$. According to Pazy \cite[Chapter 3, Theorem 1.1]{Pazy_1983}, since $\mathcal{C}$ is linear and bounded in $L^1(G)$, which follows easily by Hypothesis \ref{hypo:k}, one can associate a $C_0$-stochastic semigroup $(S_t)_{t \ge 0}$ such that, for all $f_0 \in L^1(G)$, $t \ge 0$, $S_t f_0 = f(t,\cdot)$ is the unique solution in $L^{\infty}([0,\infty); L^1(G))$ to \eqref{eq:main_pb} at time $t$. 
	
	\medskip 
	
	\textbf{Step 2: proof of \textit{i.} and \textit{ii.}} Point \textit{i.} follows from a \textit{mutatis mutandis} adaptation of the detailed proof of Mischler \cite[Theorem 1 and Corollary 1]{Mischler_1999}. The latter deals with a sole source term on the r.h.s. \hspace{-.2cm} of the equation, but, as also pointed out by Dietert-Hérau-Hutridurga-Mouhot \cite[Appendix B]{Dietert_2022}, the result can be easily extended to bounded linear operator in $L^1$, as is the case of our operator $\mathcal{C}$ under Hypothesis \ref{hypo:k}. We refer the interested reader to \cite{Mischler_1999}. 
	
	\medskip 
	
	Once the trace is well-defined, point \textit{ii.} follows from a direct computation: 
	\begin{align*}
		\frac{\d}{\d t} \int_G S_t f(x,v) \, \d v \, \d x &= - \int_G v \cdot \nabla_x S_tf(x,v) \, \d v \, \d x + \int_G \Big[ \int_{\R^d} k(x,v',v) S_tf(x,v')\,  \d v' \Big] \d v \, \d x \\
		&\qquad - \int_G S_t f(x,v) \Big[ \int_{\R^d} k(x,v,v') \, \d v' \Big] \d v \, \d x \\
		&=-\int_\Sigma (v \cdot n_x) \gamma_t f(x,v) \, \d v \, \d \zeta(x) = 0,
		\end{align*}
		where the second equality follows from Green's formula for the boundary term, and using that the two collision terms cancel out, thanks to Fubini's theorem, using Hypothesis \ref{hypo:k} and that $S_t f \in L^1(G)$. The last equality follows from the normalization property \eqref{eq:normalization_basic} and \eqref{eq:normalization_2}. 
	
	\medskip

	\textbf{Step 3: contraction property.}
	Let $t \ge 0$. By Kato's inequality, one has
	\begin{align*}
		\frac{d}{dt} \int_G |S_t f| \, \d v \, \d x &\le - \int_{G} v \cdot \nabla_x |S_t f| \,  \d v \, \d x + \int_{G} \int_{\R^d} k(x,v,v') |S_t f|(v') \, \d v' \, \d v \, \d x \\
		&\qquad - \int_G \sigma(x) |S_t f| \, \d v \, \d x \\
		&= - \int_{\partial \Omega \times \R^d} \gamma |S_t f| \, (v \cdot n_x) \,  \d v \, \d \zeta(x) 
	\end{align*}
where we used Tonelli's theorem to prove that the last two terms on the right-hand-side of the first inequality cancel out. By \textit{i.}
\begin{align*}
	\int_{\partial \Omega \times \R^d} \gamma |S_t f| |v \cdot n_x| \, \d \zeta(x) \, \d v = \int_{\partial \Omega \times \R^d} |\gamma_t f| |v \cdot n_x| \, \d \zeta(x) \, \d v,
\end{align*}
and it follows from the boundary condition and the triangle inequality that
\begin{align*}
	-\int_{\Sigma} |\gamma_t f| (v \cdot n_x) \, \d v \, \d \zeta(x) &= - \int_{\Sigma_+} |\gamma_t f| |v \cdot n_x| \, \d v \, \d \zeta(x) \\
	&\qquad + \int_{\Sigma_-} \Big| \int_{\Sigma_+^x} R(u \to v; x) |u \cdot n_x| \gamma_t f(x,u)\,  \d u \Big| |v \cdot n_x| \, \d v \, \d \zeta(x) \\
	&\le -\int_{\Sigma_+} |\gamma_t f|(x,v) \, |v \cdot n_x| \, \d v \, \d \zeta(x) \\
	&\qquad + \int_{\Sigma_+}  |\gamma_t f|(x,u) \, |u \cdot n_x| \Big[ \int_{\R^d} R(u \to v; x) |v \cdot n_x| \, \d v \Big] \d u \, \d \zeta(x) \\
	&= 0,
\end{align*}
where we also used the normalization property to obtain the last equality.

\medskip 

\textbf{Step 4. Proof of (iv)} The positivity property is a classical consequence of the contraction in $L^1$ and of the linearity, see for instance \cite[Proof of Theorem 3, Step 4]{Bernou_2020}. 
\end{proof}

	\section{Lyapunov Conditions}
	\label{sec:Lyapunov}
	
	Recall the definition of the weights $m_{\alpha}$ from \eqref{eq:def_m_alpha}, $\alpha \in (0,d)$, and that $0 < \delta \le \frac{\delta_k}{d} \wedge \frac{\varepsilon}{2d}$ is fixed throughout the paper. The goal of this section is to prove the following Lyapunov conditions:
	\begin{prop} \
		\label{prop:lyapunov}
		
		\begin{enumerate}
		\item For $\alpha \in (1, d)$, under Hypotheses \ref{hypo:boundary} and \ref{hypo:k}, there exists a constant $K > 0$ such that for all $T > 0$, all $f \in L^1_{m_\alpha}(G)$,
		\begin{align}
			\label{eq:Lyapunov_subgeom}
			\|S_T f\|_{m_\alpha} + \alpha \int_0^T  \|S_s f\|_{m_{\alpha-1}} \, \d s \le \|f\|_{m_\alpha} +  K(1+T) \|f\|_{L^1}.
			\end{align}
		\item Under Hypothesis \ref{hypo:boundary}-\ref{hypo:sigma_below}, for all $\alpha \in (1,d)$, there exists a constant $K_2 > 0$ such that for all $T > 0$, all $f \in L^1_{m_{\alpha}}(G)$,
		\begin{align}
			\label{eq:Lyapunov_expo}
			\|S_T f\|_{m_{\alpha}} + \sigma_0 \int_0^T  \|S_s f\|_{m_{\alpha}} \, \d s \le \|f\|_{m_{\alpha}} +  K_2(1+T) \|f\|_{L^1}.
		\end{align}
	\end{enumerate}
\end{prop}
	
	We will make use of both the function $\tau$, see \eqref{eq:def_tau}, and $q$ defined for all $(x,v) \in \bar G$ by
	\begin{align}
		\label{eq:def_q}
		q(x,v) := x + \tau(x,v) v. 
	\end{align}

	To derive the Lyapunov conditions, we first need to obtain some control of the flux. Using the general Cercignani-Lampis boundary condition rather than the diffuse one generates additional difficulty, see \cite[Remark 17]{Bernou_2021}. We start by deriving the following. 
	
	\begin{lemma}[Control of the flux]
		\label{lemma:control_flux}
		Under Hypotheses \ref{hypo:boundary} and \ref{hypo:k}, we have
		\begin{enumerate}
		\item[\CLBC] for all $\Lambda > 0$, there exists an explicit constant $C_{\Lambda} > 0$ s.t. for all $f \in L^1(G)$, $T > 0$,
		\begin{align}
			\label{eq:control_Flux_CL}
			\int_0^T \int_{\pO} \int_{\{v \cdot n_x > 0, |v| \le \Lambda\}} |\vp| \,\gamma_+ |S_s f|(x,v) \, \d v \, \d \zeta(x) \, \d s \le C_{\Lambda} (1 + T) \|f\|_{L^1};
		\end{align}
		\item[\MBC] there exists an explicit constant $C > 0$ such that for all $f \in L^1(G)$, $T > 0$, 
		\begin{align}
			\label{eq:control_flux_diffuse}
			\int_0^T \int_{\Sigma_+} |\vp| \,\gamma_+ |S_s f|(x,v) \, \d v \, \d \zeta(x) \, \d s \le C (1 + T) \|f\|_{L^1}. 
		\end{align}
	\end{enumerate}
	\end{lemma}
	
	\begin{proof}
		\textbf{Step 1: an inequality for a boundary term.} We have, by definition of $(S_t)_{t \ge 0}$, by linearity of \eqref{eq:main_pb} and by positivity of the semigroup, that
		\[ \partial_t |S_t f| + v \cdot \nabla_x |S_t f| = \mathcal{C} \big( |S_t f| \big), \qquad \text{a.e. in } [0,T] \times G. \]
		Recall that $x \mapsto n_x$ is a $W ^{1,\infty}(\Omega)$ map by hypothesis. Multiplying this equation by $(v \cdot n_x)$ and integrating on $[0,T] \times \Omega \times \{v \in \R^d, |v| \le 1\}$, we find
		\begin{align*}
			&\int_0^T \int_{\Omega} \int_{\{|v| \le 1\}} (v \cdot n_x) \, \Big( \partial_t + v \cdot \nabla_x \Big) |S_t f|(x,v) \, \d v \,  \d x \, \d t \\
			& \qquad = \int_0^T \int_{\Omega} \int_{\{|v| \le 1\}} (v \cdot n_x) \, \Big[ \int_{\R^d} k(x,v',v) \, |S_t f|(x,v') \, \d v' \Big] \, \d v \, \d x \, \d t \\
			&\qquad \qquad - \int_0^T \int_{\Omega} \int_{\{|v|\le 1\}} (v \cdot n_x) \, |S_t f|(x,v) \, \sigma(x) \, \d v \, \d x \, \d t
		\end{align*}
	where we used the definition of $\sigma$. 
		Integrating by parts in both time and space on the left-hand side, we find
		\begin{align}
			\label{eq:temp_small_flux}
			&\Big[ \int_{\Omega} \int_{\{|v| \le 1\}} (v \cdot n_x) \, |S_t f|(x,v)  \, \d v \, \d x \Big]_0^T  \nonumber \\
			&- \int_0^T \int_\Omega \int_{\{|v| \le 1\}} |S_t f| (x,v) \, v \cdot \nabla_x (v \cdot n_x) \, \d v \, \d x \, \d t \nonumber   \\
			&+ \int_0^T \int_{\pO} \int_{\{|v| \le 1\}} |v \cdot n_x|^2 \, \gamma |S_tf|(x,v) \,  \d v \, \d \zeta(x) \, \d t \nonumber  \\
			& \qquad = \int_0^T \int_{\Omega} \int_{\{|v| \le 1\}} (v \cdot n_x) \, \Big[ \int_{\R^d} k(x,v',v) \, |S_t f|(x,v') \, \d v' \Big]\, \d v \, \d x \, \d t \nonumber \\
			&\qquad \qquad - \int_0^T \int_{\Omega} \int_{\{|v|\le 1\}} (v \cdot n_x) \, |S_t f|(x,v) \, \sigma(x) \, \d v \, \d x \, \d t 
		\end{align}
		where we used that, according to Theorem \ref{thm:well_posed_trace}, point \textit{i.},  
		\begin{align}
			\label{eq:absolute_trace}
			|\gamma S_t f(x,v)| = \gamma |S_t f|(x,v) \qquad \text{ a.e. in } \big((0,\infty) \times \Sigma_+\big) \cup \big((0,\infty) \times \Sigma_-\big).
		\end{align}
		We note first that
		\begin{align}
			\label{eq:temp_collision_small_flux}
			&\Big| \int_0^T \int_{\Omega} \int_{\{|v| \le 1\}} (v \cdot n_x) \, \Big[ \int_{\R^d} k(x,v',v) \, |S_t f|(x,v') \, \d v' \Big] \d v \, \d x \, \d t \nonumber  \\
			&\quad  - \int_0^T \int_{\Omega} \int_{\{|v|\le 1\}} (v \cdot n_x) \, |S_t f|(x,v) \, \sigma(x) \, \d v \, \d x \, \d t \Big| \nonumber  \\
			&\qquad \le \big(k_{\infty} + \sigma_\infty \big) \int_0^T \|S_s f\|_{L^1} \, \d s,
		\end{align}
	where we used Tonelli's theorem and Hypothesis \ref{hypo:k}. 
	Isolating the integral on the boundary $\partial \Omega$ and throwing away the integral included in $\Sigma_+$ in \eqref{eq:temp_small_flux}, using that $x \mapsto n_x$ belongs to $W^{1,\infty}(\Omega)$, the triangle inequality and \eqref{eq:temp_collision_small_flux}, this leads to 
		\begin{align*}
			&\int_0^T \int_{\{ (x,v) \in \Sigma_-, |v| \le 1\}}  |\vp|^2 \,  \gamma_- |S_t f|(x,v) \, \d v \, \d \zeta(x) \, \d t\\
			 &\qquad \leq 2 \|f\|_{L^1} + C\Big( \|n_{\cdot}\|_{W^{1,\infty}} + k_{\infty} + \sigma_\infty  \Big) \int_0^T \|S_s f\|_{L^1} \, \d s \\
			&\qquad \le C (1+T) \|f\|_{L^1}, 
		\end{align*}
		where we used the $L^1$ contraction from Theorem \ref{thm:well_posed_trace}.
		
		\medskip 
		
		\textbf{Step 2: Conclusion in case \MBC.} 
		Using the boundary condition, Tonelli's theorem and the positivity
		\begin{align}
			\label{eq:temp_M}
			\int_0^T &\int_{\{ (x,v) \in \Sigma_-, |v| \le 1\}}  |\vp|^2 \,  \gamma_- |S_t f|(x,v) \, \d v\,  \d \zeta(x) \, \d t \nonumber \\
			&\ge \int_0^T \int_{\Sigma_+} \beta(x) |\up| \gamma_+ |S_t f|(x,u) \int_{\{v \in \Sigma^x_-, |v| \le 1\}} |\vp|^2 M(x,v) \, \d v\, \d u \, \d \zeta(x) \, \d t. 
		\end{align}
		Note that $x \mapsto \int_{\{v \in \Sigma_-^x, |v| \le 1\}} M(x,v) (\vp)^2 \, \d v$ is continuous and positive, since $x \mapsto M(x,v)$ and $x \mapsto n_x$ are continuous for all $v \in \R^d$. Since $\partial \Omega$ is compact, letting 
		\[ B_0= \min_{x \in \partial \Omega} \int_{\{v \in \Sigma_-^x, |v| \le 1\}} M(x,v) (\vp)^2 \, \d v > 0, \]
		we deduce from \eqref{eq:temp_M}, Step 1 and since $\beta(x) \ge \beta_0$ that 
		\begin{align*}
			\int_0^T &\int_{\Sigma_+} |\vp| \gamma_+ |S_t f|(x,v) \, \d v \, \d \zeta(x) \, \d t \le \frac{1}{B_0 \beta_0} C (1+T) \|f\|_{L^1}, 
		\end{align*}
		which concludes the proof of \eqref{eq:control_flux_diffuse}.
			
			\dd 
		\medskip

		\textbf{Step 3: Proof of \eqref{eq:control_Flux_CL}.} Using the boundary condition and Tonelli's theorem, 
		\begin{align*}
			\int_0^T &\int_{\{ (x,v) \in \Sigma_-, |v| \le 1\}}  |\vp|^2 \,  \gamma_- |S_t f|(x,v) \, \d v\,  \d \zeta(x) \, \d t \\
			&= \int_0^T \int_{\Sigma_+} |\up| \, \gamma_+ |S_t f|(x,u) \int_{\{v \in \Sigma^x_-, |v| \le 1\}} |\vp|^2 R(u \to v;x) \, \d v\, \d u \, \d \zeta(x) \, \d t
		\end{align*}
		and finally we obtain from Step 1
		\begin{align}
			\label{eq:before_J}
			\int_0^T  &\int_{\{(x,u) \in \Sigma_+, |u| \le \Lambda\}} |\up| \, \gamma_+ |S_t f|(x,u) \int_{\{v \in \Sigma_-^x, |v| \le 1\}} |\vp|^2 \, R(u\to v; x) \, \d v \, \d u \, \d \zeta(x) \, \d t \nonumber \\
			&\le \int_0^T \int_{\Sigma_+} |\up| \, \gamma_+ |S_t f|(x,u) \int_{\{v \in \Sigma^x_-, |v| \le 1\}} |\vp|^2 \, R(u \to v;x) \,  \d v \, \d u \, \d \zeta(x) \, \d t \nonumber \\
			&\leq C (1 + T) \|f\|_{L^1},
		\end{align}
		where we used that $\{ (x,u) \in \Sigma_+, |u| \le \Lambda \} \subset \Sigma_+$ and the positivity of the integrand. \dd 
		We claim that there exists $c_{\Lambda} > 0$ such that for all $(x,u) \in \Sigma_+$ with $|u| \le \Lambda$, 
		\[ J_{u,x} := \int_{\{v \in \Sigma_-^x, |v| \le 1\}} |\vp|^2 \, R(u\to v;x) \, \d v \ge c_{\Lambda}. \]
		Indeed, 
		\begin{align*}
			J_{u,x} =  &\int_{\{v \in \Sigma_-^x, |v| \le 1\}} \frac{|\vp|^2}{\theta(x) \rp (2 \pi \theta(x) \rt (2 - \rt))^{\frac{d-1}{2}}}  \exp \Big(-\frac{|\vp|^2}{2 \theta(x) \rp} -\frac{(1-\rp)|\up|^2}{2 \theta(x) \rp} \Big) \\
			& \quad \times I_0 \Big(\frac{ (1-\rp)^{\frac12} \up \cdot \vp}{ \theta(x) \rp} \Big) \exp \Big(-\frac{|\vt - (1-\rt) \ut|^2}{2 \theta(x) \rt (2 - \rt)} \Big) \, \d v,
		\end{align*} 
		and, since $x \mapsto n_x$ and $x \mapsto \theta(x)$ are continuous, $(x,u) \mapsto J_{u,x}$ is continuous with $J_{u,x} > 0$ on the compact set $\{(x,u) \in \Sigma_+, |u| \le \Lambda\}$.
		Therefore, there exists $c_{\Lambda} > 0$ such that for all $(x,u) \in \Sigma_+$ with  $|u| \le \Lambda$, 
		\begin{align*}
			J_{u,x} \ge c_{\Lambda}.
		\end{align*}
		Note that, for any given $\Lambda$, the value of $c_{\Lambda}$ can be computed explicitly.
		Inserting this into \eqref{eq:before_J}, we find 
		\begin{align*}
			c_{\Lambda} \int_0^T \int_{\{(x,v) \in \Sigma_+, |v| \le \Lambda\}} |\vp| \, \gamma_+ |S_t f|(x,v) \, \d v \, \d \zeta(x) \, \d t \le C (1 + T) \|f\|_{L^1}
		\end{align*}
		and the conclusion follows by setting $C_{\Lambda} = \frac{C}{c_{\lambda}} > 0$. 
	\end{proof}

	In the case \CLBC, we will also need the following result, whose proof in the case $\delta = \tfrac1{4}$, $\alpha \in (1,d+1)$ and $L_1 = 1$, $L_2 = d(\Omega)$ is given in \cite[Lemma 18]{Bernou_2021} and can be adapted directly to treat any small $\delta$ and any $L_1, L_2 > 0$. We also emphasize that the proof carries on as long as $\max((1-\rp), (1-\rt)^2) < 1$, which encompasses the case $\rp = 1, \rt \ne 1$ and $\rt = 1, \rp \ne 1$. 
	
	\begin{lemma}
		\label{lemma:Iux}
		Let $\tilde \delta \in (0,\tfrac12)$, $\alpha \in (1,d+1)$. Set, for $L_1, L_2 > 0$, $(x,u) \in \Sigma_+$,
		\begin{align*}
			I_{u,x,L_1,L_2} := \int_{ \Sigma_-^x} |\vp| \, \Big\{ \big(L_1 + L_2 + |v|^{2\tilde \delta}\big)^{\alpha} - \big(L_1 + |u|^{2\tilde \delta}\big)^{\alpha} \Big\} \, R(u \to v; x) \, \d v.
		\end{align*}  
		In case \CLBC, for any $L_1, L_2 > 0$, for all $P > 0$, there exists $\Lambda > 0$ such that for all $x \in \pO$, $u \in \Sigma^x_+$ with $|u| \ge \Lambda$,
		\begin{align}
			\label{eq:conclusion_lemma_flux_ineq} 
			I_{u,x,L_1,L_2} \le - P. 
		\end{align}
	\end{lemma}

\begin{proof}[Proof of Proposition \ref{prop:lyapunov}]

	It is known, see \cite[Equation (25)]{Bernou_2020} and \cite{Esposito_2013} for a detailed derivation, that for all $(x,v) \in G$, $v \cdot \nabla_x \tau(x,v) = -1$. Hence, for all $\alpha \in (1, d)$,
	\begin{align}
		\label{eq:derivative_weight} 
		&v \cdot \nabla_x m_\alpha(x,v) = - \alpha (v \cdot \nabla_x \tau(x,-v)) m_{\alpha - 1} = - \alpha \, m_{\alpha - 1}.
	\end{align}
	In the whole proof, we write Case (1) and (2) when we wish to distinguish between the proof of \eqref{eq:Lyapunov_subgeom} and the one of \eqref{eq:Lyapunov_expo}, respectively.

	\vspace{.3cm}
	
	\textbf{Step 1.} Let $\alpha \in (1,d)$, $f \in L^1_{m_\alpha}(G)$. We differentiate the $m_\alpha$-norm of $f$. First, since $n_x$ is the unit outward normal at $x \in \pO$, for $T > 0$, we apply Green's formula to find
	\begin{align}
		\label{eq:Lyapunov_base_computation}
		\frac{d}{dT} \int_G |S_T f| \, m_\alpha \, \d v \, \d x &\le \int_G |S_T f| \, (v \cdot \nabla_x m_\alpha) \, \d v \, \d x - \int_{\Sigma} (v \cdot n_x) \, m_\alpha \, (\gamma |S_T f|) \, \d v \, \d \zeta(x) \\
		&\quad + \int_G m_\alpha(x,v) \int_{\R^d} k(x,v',v) |S_T f|(x,v') \, \d v' \, \d v \, \d x \nonumber  \\
		&\quad - \int_G m_\alpha(x,v) |S_T f|(x,v) \int_{\R^d} k(x,v,v') \, \d v' \, \d v \, \d x. \nonumber 
	\end{align}
	By Theorem \ref{thm:well_posed_trace},
	\[ |\gamma S_t| f(x,v) = \gamma |S_t f|(x,v), \qquad \hbox{ a.e. in } (\R_+^* \times \Sigma_+) \cup (\R_+^* \times \Sigma_-), \]
	hence, we will not distinguish between both values in what follows.
	
	\medskip 
	
	\textbf{Step 2: First term on the r.h.s. \hspace{-.25cm} of \eqref{eq:Lyapunov_base_computation}.}
	Using \eqref{eq:derivative_weight}, we immediately obtain, for $\alpha \in (1,d)$,
		\begin{align}
			\label{eq:Lyapunov_first_term_1}
			\int_G |S_T f| (v \cdot \nabla_x m_\alpha) \, \d v \, \d x = - \alpha \|S_T f\|_{m_{\alpha-1}}.  
		\end{align}

	\medskip 
	
	\textbf{Step 3: Ante-penultimate term on the r.h.s. \hspace{-.2cm} of \eqref{eq:Lyapunov_base_computation}.}
	 By Tonelli's theorem, for any $\alpha \in (1,d)$
	\begin{align}
		\label{eq:temp_ante_penultimate}
		&\int_G m_\alpha(x,v) \int_{\R^d} k(x,v',v) |S_T f|(x,v') \, \d v' \, \d v \, \d x \nonumber \\
		&= \int_G |S_T f|(x,v) \int_{\R^d} k(x,v,v') m_\alpha(x,v') \, \d v' \, \d v \, \d x. 
			\end{align}
		We claim that 
		\begin{align}
			\label{eq:claim_int_v}
			\sup_{(x,v) \in G} \int_{\R^d} k(x,v,v') m_{\alpha}(x,v') \, \d v' < \infty. 
		\end{align} 
		
	Indeed, we clearly have, for all $(x,v) \in G$,
	\[ m_\alpha(x,v) \le \Big(e^2 + \frac{d(\Omega)}{|v|c_4} + |v|^{2\delta} \Big)^\alpha, \]
	hence, for all $K \gg 1$,
	\begin{align*}
		\int_{\R^d} k(x,v,v') m_{\alpha}(x,v') \, \d v' &\le \int_{\R^d} k(x,v,v') \Big(e^2 + \frac{d(\Omega)}{|v'|c_4} + |v'|^{2\delta} \Big)^\alpha \, \d v' \\
		&\le \int_{\{v' \in \R^d: |v'| \le K\}} k(x,v,v') \Big( e^2+K^{2\delta} + \frac{d(\Omega)}{|v'|c_4} \Big)^\alpha \, \d v' \\
		&\qquad + \int_{\{v' \in \R^d: |v'| \ge K\}} k(x,v,v') \Big( e^2+ \frac{d(\Omega)}{c_4 K} + |v'|^{2\delta} \Big)^\alpha \, \d v'.
	\end{align*}
	Since $\alpha \in (1,d)$, by convexity, $(x+y)^\alpha \le C_\alpha (x^{\alpha} + y^{\alpha})$ for some constant $C_\alpha > 0$ for all $x,y \in \R_+$. Thus,
	\begin{align*}
		\int_{\R^d} k(x,v,v') m_{\alpha}(x,v') \, \d v' &\le C_{K,\alpha,d(\Omega),c_4} \Big( \sigma(x) + \int_{\{|v'| \le K\}} k(x,v,v') \frac{1}{|v'|^\alpha} \Big) \, \d v' \\
		&\qquad \qquad  + \int_{\{|v'| \ge K\}} |v'|^{2 \delta \alpha} k(x,v,v') \, \d v' \Big) \\
		&\le C \big(\sigma_\infty + k_\infty + zM_{\delta_k} \big)
	\end{align*}
	where we use Hypothesis \ref{hypo:k}, the fact that $2 \alpha \delta \le 2 \delta_k$ by choice of $\delta$, and that, by a change to hyperspherical coordinates,
	\begin{align*}
		\int_{\{v \in \R^d, |v| \le K\}} \frac{1}{|v|^\alpha} \, \d v = C_d \int_0^K \frac{1}{r^{\alpha + 1 - d}} \d r \le C_{d,\alpha,K}
	\end{align*}
	since $\alpha + 1 - d < 1$ by assumption on $\alpha$. This concludes the proof of \eqref{eq:claim_int_v}. Injecting this into \eqref{eq:temp_ante_penultimate}, we find, for some constant $C$ independent of $T$ and $f$, using the $L^1$ contraction
	\begin{align}
		\label{eq:Lyapunov_antepenultimate}
	\int_G m_\alpha(x,v) \int_{\R^d} k(x,v',v) |S_T f|(x,v') \, \d v' \, \d v \, \d x \le C \|S_T f\|_{L^1} \le C \|f\|_{L^1}.
	\end{align} 

\medskip 

\textbf{Step 4: Control of the last term on the r.h.s. of \eqref{eq:Lyapunov_base_computation}.}
\begin{enumerate}
	\item[Case (1):] Since $k$ is non-negative, for any $\alpha \in (1,d)$, one easily obtains
	\begin{align}
		\label{eq:Lyapunov_last_term_1}
		- \int_G m_\alpha(x,v) |S_T f|(x,v) \int_{\R^d} k(x,v,v') \, \d v' \, \d v \, \d x \le 0. 
	\end{align}
\item[Case (2):] We use Hypothesis \ref{hypo:sigma_below}. We simply obtain
\begin{align}
	\label{eq:Lyapunov_last_term_2}
	- \int_G \sigma(x) m_\alpha(x,v) |S_T f|(x,v) \, \d v \, \d x \le - \sigma_0 \|S_T f\|_{m_\alpha}. 
\end{align}
\end{enumerate} 

\medskip 

\textbf{Step 5: Control of the boundary term in \eqref{eq:Lyapunov_base_computation}.}
Let 
\[B := -\int_\Sigma (v \cdot n_x) m_\alpha(x,v) \big(\gamma |S_T f|\big)\, \d v \, \d \zeta(x). \]
We show that, in case \CLBC, for some $\Lambda > 0$,
\begin{align}
	\label{eq:Lyapunov_boundary_CL}
	B \le C_{\Lambda} \int_{\{(x,v) \in \Sigma_+, |v| \le \Lambda \}} \gamma_+ |S_T f|(x,v) \, |\vp| \, \d v \, \d \zeta(x),
\end{align}
while in the case \MBC,
\begin{align}
	\label{eq:Lyapunov_boundary_Maxwell}
B \le C \int_{\Sigma_+} \gamma_+ |S_T f|(x,v) \, |\vp| \, \d v \, \d \zeta(x).
\end{align}
This step is divided into two further substeps, the first one treating \eqref{eq:Lyapunov_boundary_CL}, the second one focusing on \eqref{eq:Lyapunov_boundary_Maxwell}.

\medskip 

\textbf{Step 5.1: case \CLBC.} 
By definition of $B$,
\begin{align*}
	B &= - \int_{\Sigma_+} \gamma_+ |S_T f| \, |\vp| \, m_\alpha(x,v) \, \d v \, \d \zeta(x) + \int_{\Sigma_-} \gamma_- |S_T f| \, |\vp| \, m_\alpha(x,v) \, \d v \, \d \zeta(x)  \\
	&=: -B_1 + B_2,
\end{align*}
the last equality standing for a definition of $B_1$ and $B_2$. 
Using the boundary condition and Tonelli's theorem, it is straightforward to see that
\begin{align*}
	B_2 = \int_{\Sigma_+} \gamma_+ |S_T f|(x,u) \, |\up| \Big( \int_{\Sigma_-^x} m_\alpha(x,v) \, |\vp| \, R(u \to v; x) \, \d v \Big) \d u \, \d \zeta(x).
\end{align*}
Set, for all $x \in \pO$, $u \in \Sigma_+^x$, 
\[ P_{u,x} := \int_{\Sigma_-^x} m_\alpha(x,v) \, |\vp| \, R(u \to v; x) \, \d v. \]
We will split the integral in $P_{u,x}$ between an integral on $\{v \in \Sigma_-^x, |v| \le 1\}$ and one on the set $\{v \in \Sigma_-^x, |v| \ge 1\}$. We start with the treatment of the former. 
Note first that, for all $v \in \Sigma_-^x$, $\up \cdot \vp \le 0$ so that, using the definition of $I_0$ \eqref{eq:defI0}, 
\[ I_0 \Big( \frac{(1-\rp)^{\frac12} \up \cdot \vp}{\theta(x) \rp} \Big) \le \exp \Big(-\frac{2(1-\rp)^{\frac12} \up  \cdot \vp}{2\theta(x) \rp} \Big), \]
hence, using $\theta(x) \ge \theta_0$ for some $\theta_0 > 0$ for all $x \in \pO$ (by positivity and continuity assumptions)
\begin{align*}
	R(u \to v;x) &= \frac{\exp \Big(-\frac{|\vt - (1-\rt) \ut|^2}{2 \theta(x) \rt (2-\rt)} \Big)}{(2 \pi \theta(x) \rt (2 - \rt))^{\frac{d-1}{2}}}  \frac{\exp \Big(-\frac{|\vp|^2}{2 \theta(x) \rp} -\frac{(1-\rp)|\up|^2}{2 \theta(x) \rp} \Big)}{\theta(x) \rp} I_0 \Big( \frac{(1 - \rp)^{\frac12} \up \cdot \vp}{ \theta(x) \rp} \Big) \\
	&\le \frac{1}{(2 \pi \theta(x) \rt (2 - \rt))^{\frac{d-1}{2}}} \frac{\exp \Big(-\frac{|\vp + (1-\rp)^{\frac12} \up|^2}{2 \theta(x) \rp} -\frac{|\vt - (1-\rt) \ut|^2}{2 \theta(x) \rt (2-\rt)} \Big)}{\theta(x) \rp} \\ 
	&\le \frac{1}{\theta_0 \rp (2 \pi \theta_0 \rt (2-\rt))^{\frac{d-1}2}} \le C,
\end{align*}
where we used the upper bound $1$ for both exponentials. We clearly have, for all $(x,v) \in \bar G$,
\[ m_\alpha(x,v) \le \Big(e^2 + \frac{d(\Omega)}{c_4|v|} + |v|^{2\delta} \Big)^\alpha \]
by definition of $m_\alpha$, and using that $|\vp| \le |v|$, we get
\begin{align*}
	\int_{\{v \in \Sigma_-^x, |v| \le 1\}} \hspace{-.8cm} m_\alpha(x,v) \, |\vp| \, R(u \to v; x) \, \d v &\le \int_{\{v \in \Sigma_-^x, |v| \le 1\}} \Big( e^2 + 1 + \frac{d(\Omega)}{|v|c_4} \Big)^{\alpha} \, |\vp| \, R(u \to v; x) \, \d v \\
	&\le C \int_{\{v \in \Sigma_-^x, |v| \le 1\}} \Big( e^2 + 1 + \frac{d(\Omega)}{c_4 |v|} \Big)^{\alpha} \, |v| \, \d v \\
	&\le C_{\alpha}
\end{align*}
for some constant $C_{\alpha} > 0$ independent of $u$ and $x$. We used that $\alpha < d$ to obtain the existence of such finite $C_{\alpha}$ (as can be checked by using an hyperspherical change of variable, see Step 3). On the other hand,
\begin{align*}
	\int_{\{v \in \Sigma_-^x, |v| \ge 1\}} &m_\alpha(x,v) \, |\vp| \, R(u \to v; x) \, \d v \\
	&\le \int_{\{v \in \Sigma_-^x, |v| \ge 1\}} \Big(e^2 + \frac{d(\Omega)}{c_4} + |v|^{2\delta}\Big)^{\alpha} \, |\vp| \, R(u \to v; x) \, \d v \\
	&\le \int_{\Sigma_-^x} \Big(e^2 + \frac{d(\Omega)}{c_4} + |v|^{2\delta}\Big)^{\alpha} \, |\vp| \, R(u \to v; x) \, \d v. \end{align*}
Overall, we proved that 
\begin{align}
	\label{eq:JU}
	P_{u,x} \le C_{\alpha} + \int_{\Sigma_-^x} \Big(e^2 + \frac{d(\Omega)}{c_4} + |v|^{2\delta}\Big)^{\alpha} \, |\vp| \, R(u \to v; x) \, \d v.
\end{align}
On the other hand, for all $(x,u) \in \Sigma_+$, $\int_{\Sigma_-^x} |\vp| \, R(u\to v; x) \, \d v = 1$ by \eqref{eq:normalization_basic} and \eqref{eq:normalization_2}. Since we also have $\tau(x,-v) \le d(\Omega)/|v|$ and $c_4 < 1$, $\frac{d(\Omega)}{c_4 |v|} - \tau(x,-v) \ge 0$, so that 
\[ m_\alpha(x,u) \ge (e^2 + |u|^{2\delta})^{\alpha},\]
we get
\begin{align}
	\label{eq:B1}
	- B_1 \le - \int_{\Sigma^+} |\up| \, \gamma_+  |S_T f|(x,u)\,  (e^2+|u|^{2\delta})^{\alpha} \int_{\Sigma_-^x} |\vp| \, R(u \to v; x) \, \d v  \, \d u \, \d \zeta(x). 
\end{align}
Gathering \eqref{eq:JU}, \eqref{eq:B1} and the definition of $B$, we find
\begin{align*}
	B &\le \int_{\Sigma_+} |\up| \, \gamma_+ |S_T f|(x,u) \, \\
	&\quad \times \Big\{ C_{\alpha} + \int_{\Sigma_-^x} \Big[\Big(e^2 + \frac{d(\Omega)}{c_4} + |v|^{2\delta}\Big)^{\alpha} - (e^2 + |u|^{2\delta})^{\alpha} \Big] \, |\vp| \, R(u \to v; x) \, \d v \Big\} \, \d u \, \d \zeta(x) \\
	&\le \int_{\Sigma_+} |\up| \, \gamma_+ |S_T f|(x,u) \, \big( C_{\alpha} + I_{u,x,e^2, \frac{d(\Omega)}{c_4}} \big)\, \d u \, \d \zeta(x),
\end{align*}
where $I_{u,x,e^2, d(\Omega)/c_4}$ is defined as in Lemma \ref{lemma:Iux} with $\tilde \delta = \delta$, $\alpha$ as before and $L_1 = e^2$, $L_2 =  \frac{d(\Omega)}{c_4}$. Splitting $\Sigma_+^x$ as
\[ \Sigma_+^x = \Big\{u \in \Sigma_+^x: |u| < \Lambda\Big\} \cup \Big\{u \in \Sigma_+^x: |u| \ge \Lambda\Big\} \]
with $\Lambda > 0$ given by Lemma \ref{lemma:Iux} applied with $P = C_{\alpha}$, we find that 
\begin{align}
	\label{eq:claim_large_lambda}
	 \int_{\{(x,u) \in \Sigma_+, |u| \ge \Lambda\}} |\up| \, |\gamma_+ S_T f|(x,u) \, \big( C_{\alpha} + I_{u,x,e^2, \frac{d(\Omega)}{c_4}} \big) \d u \, \d \zeta(x) \le 0, 
\end{align}
leading to 
\begin{align}
	\label{eq:tmp_B}
	B &\le \int_{\pO} \int_{\{u \in \Sigma_+^x, |u| \le \Lambda\}} |\up| \, |\gamma_+ S_T f|(x,u) \Big(C_{\alpha} + I_{u,x,e^2, \frac{d(\Omega)}{c_4}} \Big) \d u \, \d \zeta(x)  \nonumber  \\
	&\le \int_{\pO} \int_{\{u \in \Sigma_+^x, |u| \le \Lambda\}} |\up| \, |\gamma_+ S_T f|(x,u) \nonumber \\
	&\qquad \qquad \times \Big( C_{\alpha} + \int_{\Sigma_-^x} \Big(e^2 + \frac{d(\Omega)}{c_4} + |v|^{2\delta}\Big)^{\alpha} \, |\vp| \, R(u\to v;x) \, \d v \Big) \d u \, \d \zeta(x).  
\end{align}
We claim that
\begin{align}
	\label{eq:claim_low_Lambda}
	 \sup_{x \in \partial \Omega, u \in \Sigma_+^x, |u| \le \Lambda} \int_{\Sigma_-^x}\Big(e^2 + \frac{d(\Omega)}{c_4} + |v|^{2\delta}\Big)^{\alpha} \, |\vp| \, R(u \to v; x) \, \d v \in (0,\infty). 
	 \end{align}
	 
Note that the proof of \eqref{eq:Lyapunov_boundary_CL} directly follows from this claim and \eqref{eq:claim_large_lambda}.  It thus only remains to prove \eqref{eq:claim_low_Lambda}. 
Let $x \in \partial \Omega$, $u \in \Sigma_+^x$ with $|u| \le \Lambda$. For simplicity we will rely on the probabilistic tools introduced in \cite[Section 2.2]{Bernou_2021}. We may write the integral inside the supremum as 
\[ \EE \Big[ \Big(e^2 + \frac{d(\Omega)}{c_4} + \big(|X|^2 + |Y|^2\big)^{\delta} \Big)^{\alpha}\Big], \]
for $Y \sim \mathrm{Ri}((1-\rp)^{\frac12}|\up|, \theta(x) \rp)$ a Rice distribution (see \cite[Definition 12]{Bernou_2021}) of parameters $(1-\rp)^{\frac12}|\up|$ and $\theta(x) \rp$,  and $X \sim \mathcal{N}((1-\rt) \ut, \theta(x) \rt (2 - \rt) I_{d-1})$ a Gaussian random variable, where $I_k$ denotes the identity matrix of size $k \times k$. Using \cite[Proposition 13]{Bernou_2021} (taken from \cite{Kobayashi_2009}), we have, for any $\vartheta \in [0,2\pi)$, 
\[ Y \overset{\mathcal{L}}{=} \sqrt{Y_1^2 + Y_2^2} \quad \text{  with } \]
\[ Y_1 \sim \mathcal{N}\Big((1-\rp)^{\frac12} |\up| \cos(\vartheta), \theta(x) \rp\Big), \qquad Y_2 \sim \mathcal{N}\Big((1-\rp)^{\frac12} |\up| \sin(\vartheta), \theta(x) \rp\Big) \] 
two random variables independent of everything else. The integral thus rewrites 
\begin{align*}
	\mathbb{E} \Big[ \Big(e^2 + \frac{d(\Omega)}{c_4} + \big( |X|^2 + |Y_1|^2 + |Y_2|^2 \big)^\delta \Big)^\alpha \Big],
\end{align*} and is finite for any $(x,u)$ in $\Sigma_+$ with $|u| \le \Lambda$ by standard property of the moments of Gaussian random variables. 

Since $x \mapsto n_x$ and $x \mapsto \theta(x)$ are continuous, for $(x_n, u_n) \to (x,u)$ in $\{(x,u) \in \Sigma_+, |u| \le \Lambda\}$, $v \in \R^d$, 
\begin{align*}
	\lim \limits_{n \to \infty} \mathbf{1}_{\{v \cdot n_{x_n} < 0\}} |v \cdot n_{x_n}| R(u_n \to v; x_n) = \mathbf{1}_{\{v \cdot n_{x} < 0\}} |v \cdot n_{x}| R(u \to v; x)
\end{align*}
almost everywhere, hence 
\begin{align*}
	(x,u) \mapsto \int_{\Sigma_-^x}\Big(e^2 + \frac{d(\Omega)}{c_4} + |v|^{2\delta}\Big)^{\alpha} \, |\vp| \, R(u \to v; x) \, \d v
\end{align*}
is continuous by dominated convergence theorem. 
Since $\{(x,u) \in \Sigma_+, |u| \le \Lambda\}$ is compact, the proof of \eqref{eq:claim_low_Lambda} is complete.

\medskip 

\textbf{Step 5.2: case \MBC.} We prove here \eqref{eq:Lyapunov_boundary_Maxwell}. Using the boundary condition, we have 
\begin{align*}
	B &= \int_{\Sigma_-} |\vp| m_\alpha(x,v) \beta(x) M(x,v) \Big( \int_{\Sigma_+^x}  |v' \cdot n_x| |S_T f|(x,v') \, \d v' \Big) \, \d v \, \d \zeta(x) \\
	&\qquad + \int_{\Sigma_-} |\vp| (1-\beta(x)) m_\alpha(x,v) |S_T f|(x,\eta_x(v)) \, \d v\, \d \zeta(x) \\
	&\qquad  - \int_{\Sigma_+} |\vp| m_\alpha(x,v) |S_T f| \, \d v \, \d \zeta(x) \\
	&\le \int_{\Sigma_-} |\vp| m_\alpha(x,v) \beta(x) M(x,v) \Big( \int_{\Sigma_+^x}  |v' \cdot n_x| |S_T f|(x,v') \, \d v' \Big) \, \d v \, \d \zeta(x) \\
	&\qquad + \int_{\Sigma_+} |\vp| |S_T f|(x,v) \Big( (1-\beta_0) m_\alpha(x,\eta_x(v)) - m_\alpha(x,v) \Big) \, \d v\, \d \zeta(x),
\end{align*}
where we used the change of variable $v \to \eta_x(v)$ with Jacobian $1$ and that  $|\vp| = |\eta_x(v) \cdot n_x|$ to get the inequality. 
Let us focus on the last term on the right-hand side of this inequality. For all $(x,v) \in \Sigma_+$, since $c_4 \in (0,1)$, $(1-c_4)^4 = (1-\beta_0)$, $\alpha < d < 4$, $\tau(x,-\eta_x(v)) = 0$ and $|\eta_x(v)|=|v|$,
\begin{align*}
	(1-\beta_0)m_\alpha(x,\eta_x(v)) &= (1-c_4)^4 \Big( e^2 + \frac{d(\Omega)}{|v|c_4} + |v|^{2\delta} \Big)^\alpha\\
	& \le \Big( (1-c_4) \big(  e^2 + \frac{d(\Omega)}{|v|c_4} + |v|^{2\delta} \big) \Big)^\alpha \\
	&\le \Big( e^2 + \frac{d(\Omega)}{c_4|v|} - \frac{d(\Omega)}{|v|} + |v|^{2\delta} \Big)^\alpha\\
	&\le \Big(e^2 + \frac{d(\Omega)}{c_4 |v|} - \tau(x,-v) + |v|^{2\delta} \Big)^\alpha = m_\alpha(x,v), 
\end{align*}
where we used $\tau(x,-v) \le d(\Omega)/|v|$ by definition of $d(\Omega)$. We get, as a first conclusion,
\begin{align*}
	B &\le \int_{\Sigma_-} |\vp| m_\alpha(x,v) \beta(x) M(x,v) \Big( \int_{\Sigma_+^x}  |v' \cdot n_x| |S_T f|(x,v') \, \d v' \Big) \, \d v \, \d \zeta(x) \\
	&\le \int_{\Sigma_+} |\vp| |S_T f|(x,v) \Big( \int_{\Sigma_-^x} |\up| m_\alpha(x,u) M(x,u) \, \d u \Big) \, \d v \, \d \zeta(x).
\end{align*}
By \eqref{eq:property_M}, $\sup_{(x,v) \in \Sigma_-} M(x,v) \le \chi_M < \infty$. Combining this with the moment condition \eqref{eq:moment_M} -- noticing $1 + 2 \delta \alpha \le 1+\varepsilon$ by choice of $\delta, \alpha$ -- one can split the integral over $v\in\R^d$ exactly as in the proof of \eqref{eq:claim_int_v}. This yields
\begin{align*}
	\sup_{x \in \partial \Omega} \int_{\Sigma_-^x} |v' \cdot n_x| M(x,v') m_\alpha(x,v') \, \d v' \le C, 
\end{align*}
which concludes the proof of \eqref{eq:Lyapunov_boundary_Maxwell}. 

\medskip 

\textbf{Step 6: conclusion under Hypotheses \ref{hypo:boundary} and \ref{hypo:k}.} Using \eqref{eq:Lyapunov_first_term_1}, \eqref{eq:Lyapunov_antepenultimate}, \eqref{eq:Lyapunov_last_term_1} and Step 5 inside \eqref{eq:Lyapunov_base_computation}, we obtain, for all $\alpha \in (1,d)$, for $\Lambda > 0$ given by Step 5, for some constant $C > 0$ allowed to depend on $\alpha, d, \Lambda$, 
\begin{align*}
	\frac{\d}{\d T} \|S_T f\|_{m_\alpha} \le - \alpha \|S_T f\|_{m_{\alpha-1}} + C \|f\|_{L^1} + C B_T,
\end{align*}
where, 
\begin{align*}
	B_T= \left\{ \begin{array}{lll} 
	&\int_{\Sigma_+} |\vp| \gamma_+ |S_T f|(x,v) \, \d v \, \d \zeta(x) \qquad &\hbox{ case \MBC}, \\
	&\int_{\{(x,v) \in \Sigma_+, |v| \le \Lambda\}} \gamma_+ |S_T f|(x,v) |\vp| \, \d v \, \d \zeta(x) \qquad &\hbox{ case \CLBC. }
	\end{array}\right.
	\end{align*}
Integrating this inequality on $[0,T]$, we find
\begin{align*}
	\|S_T f\|_{m_\alpha} + \alpha \int_0^T \|S_s f\|_{m_{\alpha-1}} \, \d s \le \|f\|_{m_{\alpha}} + CT \|f\|_{L^1} + C \int_0^T B_s \, \d s.
\end{align*}
The conclusion follows by noticing that Lemma \ref{lemma:control_flux} implies, in both cases \MBC \space and \CLBC
\begin{align*}
	\int_0^T B_s \, \d s \le C(1+T)\|f\|_{L^1}. 
\end{align*}

\medskip 

\textbf{Step 7: conclusion under Hypotheses \ref{hypo:boundary}-\ref{hypo:sigma_below}.} 
 Using \eqref{eq:Lyapunov_first_term_1}, \eqref{eq:Lyapunov_antepenultimate}, \eqref{eq:Lyapunov_last_term_2} and Step 5 inside \eqref{eq:Lyapunov_base_computation}, throwing away the negative term from \eqref{eq:Lyapunov_first_term_1}, for all $\alpha \in (1,d)$, for $\Lambda > 0$ given by Step 5, for some constant $C > 0$ allowed to depend on $\alpha, d, \Lambda$, 
\begin{align*}
	\frac{\d}{\d T} \|S_T f\|_{m_\alpha} \le - \sigma_0 \|S_T f\|_{m_{\alpha}} + C \|f\|_{L^1} + C B_T,
\end{align*}
where again,
\begin{align*}
	B_T= \left\{ \begin{array}{lll} 
		&\int_{\Sigma_+} |\vp| \gamma_+ |S_T f|(x,v) \, \d v \, \d \zeta(x) \qquad &\hbox{ case \MBC}, \\
		&\int_{\{(x,v) \in \Sigma_+, |v| \le \Lambda\}} \gamma_+ |S_T f|(x,v) |\vp| \, \d v \, \d \zeta(x) \qquad &\hbox{ case \CLBC}.
	\end{array}\right.
\end{align*}
Integrating this inequality on $[0,T]$, we find
\begin{align*}
	\|S_T f\|_{m_\alpha} + \sigma_0 \int_0^T \|S_s f\|_{m_{\alpha}} \, \d s \le \|f\|_{m_{\alpha}} + CT \|f\|_{L^1} + C \int_0^T B_s \, \d s.
\end{align*}
The conclusion follows again by noticing that Lemma \ref{lemma:control_flux} implies
\begin{align*}
	\int_0^T B_s \, \d s \le C(1+T)\|f\|_{L^1}.  \qquad \qquad \qedhere
\end{align*}
\end{proof} 

\section{Proof of Theorems \ref{thm:main_1} and \ref{thm:main_2}}

\label{sec:Doeblin}

In this section, we first prove a Doeblin-Harris condition, that, along with the Lyapunov conditions obtained in Section \ref{sec:Lyapunov}, provide the proof of Theorem \ref{thm:main_1}. Theorem \ref{thm:main_2} then follows by a usual Cauchy sequence argument, and by applying Theorem \ref{thm:main_1}. Recall the definitions of $\tau$ and $q$ from \eqref{eq:def_tau} and \eqref{eq:def_q} respectively. 

\subsection{Doeblin-Harris condition}
\label{subsec:Doeblin}

We start with the proof of the Doeblin-Harris condition. 
The key point is that, since $\sigma \in L^\infty(\Omega)$, one can find a natural lower bound of the dynamics by the free-transport ones. For this, we establish first a Duhamel formula.  

\begin{lemma}
	For all $f \in L^1(G)$, for all $(t,x,v) \in \R_+ \times \bar G$, the following formula holds:
	\begin{align}
		\label{eq:Duhamel}
		S_t f(x,v) &= \mathbf{1}_{\{\tau(x,-v) \le t\}} e^{-\int_{t - \tau(x,-v)}^{t} \sigma(x-(t-s)v) \, \d s} \Big(S_{t - \tau(x,-v)} f\Big) \big(q(x,-v), v\big) \\
		&\quad + \mathbf{1}_{\{t < \tau(x,-v)\}} e^{-\int_0^t \sigma(x-(t-s)v) \, \d s } f(x-tv, v) \nonumber \\
		&\quad + \int_{\max(0, t - \tau(x,-v))}^t e^{-\int_s^{t} \sigma(x-(t-u)v) \, \d u} \nonumber \\
		&\qquad \times \int_{\R^d} \Big[ k\big(x-(t-s)v, v',v\big) \, S_s f\big(x-(t-s)v, v'\big) \Big] \d v' \, \d s. \nonumber 
	\end{align}
	
	As a consequence, for all $f \in L^1(G)$ with $f \ge 0$, for all $(x,v) \in \bar G$,  
	\begin{align}
		\label{eq:upper_bound_Duhamel}
		S_t f(x,v) \ge \mathbf{1}_{\{\tau(x,-v) \le t\}} e^{-\int_0^{\tau(x,-v)} \sigma(x-sv) \, \d s} S_{t - \tau(x,-v)} f \big(q(x,-v), v\big).
	\end{align}
\end{lemma}

\begin{proof}
	For $(x,v) \in \Sigma_- \cup \Sigma_0$, we have $\tau(x,-v) = 0$ and $q(x,-v) = x$ so that the formula is obviously true. 
	
	\medskip 
	
	\textbf{Step 1.} Consider the problem 
	\begin{align}
		\label{eq:pb_Duhamel}
		\left\{ \begin{array}{lll} 
		&\partial_t g + v \cdot \nabla_x g + \sigma(x) g = 0, \qquad &\hbox{ in } \R_+ \times G, \\
		&\gamma_- g = K \gamma_+ g, \qquad &\hbox{ on } \R_+ \times \Sigma_-, \\
		&g(0,x,v) = g_0(x,v), \qquad &\hbox{ in } G, 
		\end{array} \right.
	\end{align}
	with $K$ given by \eqref{eq:def_K}. Problem \eqref{eq:pb_Duhamel} is a bounded perturbation of the corresponding free-transport problem, and therefore (see the proof of Theorem \ref{thm:well_posed_trace}, Step 1) admits a unique solution $g$ such that for all $g_0 \in L^1(G)$, $t \ge 0$, $g(t,\cdot)$ is the unique solution in $L^{\infty}([0,\infty); L^1(G))$ to the equation taken at time $t$. We write $(e^{t \mathcal{T}})_{t \ge 0}$ for the associated $C_0$-stochastic semigroup. 
	Assume first that $(e^{t\mathcal{T}})_{t \ge 0}$ satisfies, for $(t,x,v) \in \R_+ \times \bar G$,
	\begin{align}
		\label{eq:claim_Duhamel}
		\big(e^{t\mathcal{T}} &g\big)(x,v) = e^{-\int_0^t \sigma(x-(t-s)v) \, \d s} g(x-tv,v) \mathbf{1}_{\{t < \tau(x,-v)\}} \\
		&\quad + e^{-\int^t_{t-\tau(x,-v)} \sigma(x - (t-s)v) \, \d s} \Big(e^{(t-\tau(x,-v)) \mathcal{T}} g \Big) \big(x - \tau(x,-v)v, v\big) \mathbf{1}_{\{t \ge \tau(x,-v)\}}. \nonumber 
	\end{align}
	Adding the source operator $\mathcal{C}_+$ and setting $s = \max(0,t- \tau(x,-v))$, we obtain a solution of the form
	\begin{align*}
		f(t,x,v) = \Big[e^{(t-s) \mathcal{T}} f(s,\cdot, \cdot) \Big] (x,v) + \int_s^t \Big[e^{(t-u) \mathcal{T}} \mathcal{C}_+ f(u,\cdot,\cdot) \Big] (x,v) \, \d u
	\end{align*}
	which rewrites as 
	\begin{align*}
		f(t,x,v) &= e^{-\int_s^t \sigma(x-(t-u)v) \, \d u} f(s,x-(t-s)v, v) \\
		&\qquad + \int_s^t e^{-\int_u^t \sigma(x-(t-r)v) \, \d r} \, \int_{\R^d} k\big(x-(t-u)v,v',v\big) f(u, x - (t-u)v, v') \, \d v' \, \d u. 
	\end{align*}
	Expanding on the two possible values of $\max(0,t-\tau(x,-v))$ and recalling that $q(x,-v) = x - \tau(x,-v)v$ concludes the proof. Moreover, \eqref{eq:upper_bound_Duhamel} follows by using that $(S_t)_{t \ge 0}$ is a non-negative semigroup and a change of variable in the integral inside the exponential.  
	
	\medskip 
	
	\textbf{Step 2.} We prove \eqref{eq:claim_Duhamel}. We keep, for all $t \ge 0$, the notation $g(t,\cdot,\cdot)$ for the unique solution at time $t$ of \eqref{eq:pb_Duhamel} in the remaining part of the proof. Note that, as a solution in $L^1(G)$, $g$ solves \eqref{eq:pb_Duhamel} in the sense of distributions.
	To prove \eqref{eq:claim_Duhamel}, we consider a test function $\phi \in C^{\infty}_c([0,\infty) \times \bar G)$. Then
	\begin{align*}
		&\int_0^\infty \int_G \phi(t,x,v) g(t,x,v) \, \d v \, \d x \, \d t  \\
		& \quad = \int_0^\infty \int_G \phi(t,x,v) \int_{\max(0, t-\tau(x,-v))}^t \frac{d}{ds} \Big[g\big(s, x-(t-s)v, v \big) e^{-\int_s^t \sigma(x-(t-u)v) \d u} \Big] \d s \, \d v \, \d x \, \d t \\
		&\qquad +  \int_0^\infty \int_G \phi(t,x,v) \, g\Big(\max\big(0,t-\tau(x,-v)\big), x - \big(t-\max(0, t-\tau(x,-v)\big))v, v \Big) \\
		&\qquad \hspace{7cm} \times e^{-\int_{\max(0,t-\tau(x,-v))}^t \sigma(x-(t-u)v) \, \d u} \, \d v \, \d x \, \d t.
		\end{align*}	
	Expanding the bracket in the first term on the right-hand side gives
	\begin{align*}
	&\frac{d}{ds} \Big[g\big(s, x-(t-s)v, v \big) e^{-\int_s^t \sigma(x-(t-u)v) \d u} \Big] \\
	&\quad = \Big( \big(\partial_s + v \cdot \nabla_x + \sigma \big) g \Big) (s, x - (t-s)v, v) e^{-\int_s^t \sigma(x-(t-u)v) \d u} = 0,
	\end{align*}
	since $g$ is a solution in the sense of distributions of \eqref{eq:pb_Duhamel}. This concludes the proof of \eqref{eq:claim_Duhamel} in the sense of distributions and the conclusion in $L^1$ follows by density.
	\end{proof}

Recall that $\delta \in (0, \frac{\delta_k}{d} \wedge \frac{\varepsilon}{2d})$ is fixed and set, for all $(x,v) \in \bar G$, $\langle x, v \rangle := (1 + \tau(x,v) + |v|^{2\delta})$. Our first Doeblin-Harris condition is the following.

\begin{thm}
	\label{thm:Doeblin-Harris}
	Under Hypotheses \ref{hypo:boundary} and \ref{hypo:k}, for any $\Lambda \ge 2$, there exist $T(\Lambda) > 0$ and a non-negative measure $\nu$ on $G$, depending on $\Lambda$, with $\nu \not \equiv 0$, such that for all $(x,v) \in G$, for all $f_0 \in L^1(G)$, $f_0 \ge 0$, 
	\begin{align}
		\label{eq:Doeblin-Harris}
		S_{T(\Lambda)} f_0(x,v) \ge \nu(x,v) \int_{\{(y,w) \in G, \langle y, w \rangle \le \Lambda\}} f_0(y,w) \, \d y \, \d w.
	\end{align}
	Moreover, $\nu$ satisfies $\langle \nu \rangle \le 1$ and there exists $\xi > 0$ such that for all $\Lambda \ge 2$, $T(\Lambda) = \xi \Lambda$.  
\end{thm}

The proof follows from a direct adaptation of the ones of \cite[Theorem 21]{Bernou_2021} for case \CLBC \space and of \cite[Theorem 4.2]{Bernou_2020} in case \MBC. We only give the first step to emphasize the modification.

\begin{proof}[Sketch of proof]
For all $t > 0$, $(x,v) \in \bar{G}$, we write $f(t,x,v) = S_t f_0(x,v)$. For the sake of simplicity we simply write $f(t,x,v)$ for $\gamma f(t,x,v)$ for $(t,x,v) \in \R_+ \times \Sigma$.

\vspace{.3cm}

We let $(t,x,v) \in (0,\infty) \times G$ and compute a first lower-bound for $f(t,x,v)$. Recall the definitions of $\tau$ from \eqref{eq:def_tau} and $q$ from \eqref{eq:def_q}. By \eqref{eq:upper_bound_Duhamel}, we have
\[ f(t,x,v) \ge  e^{-\int_0^{\tau(x,-v)} \sigma(x-sv) \, \d s} f(t - \tau(x,-v),q(x,-v),v) \mathbf{1}_{\{t \ge \tau(x,-v)\}}. \]
Set $y_0 = q(x,-v)$, $\tau_0 = \tau(x,-v)$. We have, using the boundary conditions and \eqref{eq:upper_bound_Duhamel} again,
\begin{align*}
	f(t,x,v) &\ge \mathbf{1}_{\{\tau_0 \le t\}} e^{-\sigma_\infty \tau_0} f(t - \tau_0, y_0,v) \\
	&\ge \mathbf{1}_{\{\tau_0 \le t\}} e^{-\sigma_\infty  \tau_0} \int_{\Sigma_+^{y_0}} f(t- \tau_0,y_0,v_0) \, |v_0 \cdot n_{y_0}| \, R(v_0 \to v; y_0) \, \d v_0 \\
	&\ge \mathbf{1}_{\{\tau_0 \le t\}} e^{-\sigma_\infty  \tau_0} \int_{\Sigma_+^{y_0}} e^{-\sigma_\infty  \tau(y_0, -v_0)} f(t- \tau_0 - \tau(y_0,-v_0),q(y_0,-v_0),v_0) \, \,\\
	&\qquad \times  \mathbf{1}_{\{\tau_0 + \tau(y_0,-v_0) \le t\}} |v_0 \cdot n_{y_0}| \, R(v_0 \to v; y_0) \, \d v_0 \\
	&\ge \mathbf{1}_{\{\tau_0 \le t\}}  e^{-\sigma_\infty  \tau_0} \int_{\Sigma_+^{y_0}} \mathbf{1}_{\{\tau_0 + \tau(y_0,-v_0) \le t\}} \,  |v_0 \cdot n_{y_0}| \, R(v_0 \to v; y_0) \, \\
	&\qquad \times e^{-\sigma_\infty  \tau(y_0, -v_0)} \int_{\Sigma_+^{q(y_0,-v_0)}} |v_1 \cdot n_{q(y_0,-v_0)}| \, R\big(v_1 \to v_0; q(y_0,-v_0)\big) \\ &\qquad  \times f(t- \tau_0 - \tau(y_0,-v_0),q(y_0,-v_0),v_1) \, \d v_1 \, \d v_0.
\end{align*}

From there, the proof in case \CLBC \space (including in the case $(\rp, \rt) = (1,1)$) is a straightforward adaptation of \cite[Proof of Theorem 21]{Bernou_2021}, the only difference being the presence of extra constants $e^{-\sigma_\infty \tau_i}$ for various time intervals $\tau_i$ appearing from the repeated use of the Duhamel formula \eqref{eq:Duhamel}. Those are easy to treat since the proof ultimately uses a truncation of the space of integration of those times on a finite interval. 

\noindent For the case \MBC, the proof follows similarly from a direct adaptation of \cite[Proof of Theorem 4.2]{Bernou_2020}, as the only features of $M$ used there are the continuity on $\partial \Omega \times \R^d$ and the positivity, both granted by assumptions. 
\end{proof} 


\begin{rmk}[Constructive property of $\nu$]
	\label{rmk:nu_constructive}
	As in \cite[Remark 22]{Bernou_2021} and \cite[Remark 8]{Bernou_2020}, even though some compactness arguments are used in the previous proof, constructive lower bounds can be derived at least in the simple case where $\Omega$ is the unit disk, see \cite[Remark 8]{Bernou_2020}. Note that the control of the additional factors due to the jumps (the ones such as $e^{-\sigma_\infty \tau_0}$) are explicit and do not rely on a compactness argument. In general, we thus expect to be able to find a constructive lower bound for any given $\Omega$.
\end{rmk} 

The following corollary allows us to relate the Doeblin-Harris condition with the weights used in Section \ref{sec:Lyapunov}.

\begin{coroll}
	\label{coroll:Doeblin-Harris}
	Under Hypotheses \ref{hypo:boundary} and \ref{hypo:k}, there exists $\Lambda_0 > 0$ such that for all $\Lambda \ge \Lambda_0$, there exist $T(\Lambda) > 0$ and a non-negative measure $\nu$ on $G$, depending on $\Lambda$, with $\nu \not \equiv 0$, such that for all $(x,v) \in G$, for all $f_0 \in L^1(G)$, $f_0 \ge 0$, 
	\begin{align}
		\label{eq:Doeblin-Harris_coroll}
		S_{T(\Lambda)} f_0(x,v) \ge \nu(x,v) \int_{D_\Lambda} f_0(y,w) \, \d y \, \d w,
	\end{align}
	where $D_\Lambda = \{(y,w) \in G: m_1(y,w) \le \Lambda \}$.
	Moreover, $\nu$ satisfies $\langle \nu \rangle \le 1$ and there exists $\xi > 0$ such that for all $\Lambda \ge 2$, $T(\Lambda) = \xi \Lambda$. 
\end{coroll}

\begin{proof}
	For all $x \in G$, $m_1(x,v) \to \infty$ as $|v| \to \infty$. Hence there exists $\Lambda_0 > 0$ such that, denoting by $\lambda$ the Lebesgue measure on $\R^d \times \R^d$, $\lambda\{(y,w) \in G, m_1(y,w) \le \Lambda_0\} > 0$. Since $c_4 < 1$ and 
	\begin{align*}
		\tau(x,v) + \tau(x,-v) \le \frac{d(\Omega)}{|v|} \le \frac{d(\Omega)}{|v|c_4}
	\end{align*}
	by definition of $\tau$, we have
	\begin{align*}
		m_1(x,v) = \Big(e^2 + \frac{d(\Omega)}{|v|c_4} - \tau(x,-v) + |v|^{2\delta} \Big) \ge \Big( 1 + \tau(x,v) + |v|^{2\delta} \Big) = \langle x,v \rangle.
	\end{align*}
	Hence, for $\Lambda \ge \Lambda_0$, we have $D_\Lambda \subset \{(x,v) \in G: \langle x, v \rangle \le \Lambda\}$ and $D_\Lambda \ne \emptyset$. The conclusion then follows from Theorem \ref{thm:Doeblin-Harris}. 
	\end{proof}

\subsection{Proof of Theorem \ref{thm:main_1}}
\label{subsec:proof_thm_1}

From the Lyapunov conditions, Proposition \ref{prop:lyapunov} and the Doeblin-Harris condition, Corollary \ref{coroll:Doeblin-Harris}, the proof of Theorem \ref{thm:main_1} follows from Harris-type theorems. We assume for simplicity that $g \equiv 0$, so that $f \in L^1_{m_\alpha}(G)$ with $\langle f \rangle = 0$ in what follows.

\vspace{.3cm} 

More precisely, the demonstration of the polynomial result \eqref{eq:polynomial_cvg_difference}  is obtained exactly as for the free-transport case \cite[Section 5]{Bernou_2021}: note that this proof only uses that the semigroup is stochastic, that the weights considered are superlinear, and that the subgeometric Lyapunov inequality \eqref{eq:Lyapunov_subgeom} and the Doeblin-Harris condition \eqref{eq:Doeblin-Harris_coroll} hold. In fact, the only difference is that we use, for $(x,v) \in G$, $\alpha \in (1,d)$, weights of the form
\[ m_\alpha(x,v) = \Big(e^2 + \frac{d(\Omega)}{|v| c_4}- \tau(x,-v) + |v|^{2\delta} \Big)^\alpha \]
rather than $w_\alpha(x,v) = (1 + \tau(x,v) + |v|^{2\delta})^\alpha$, but this difference is not seen at the level of the proof, which only uses the asymptotic behavior of the weights as $|v| \to +\infty$ once \eqref{eq:Lyapunov_subgeom} and Corollary \ref{coroll:Doeblin-Harris} are established.  

Alternatively, the polynomial result can be obtained by applying \cite[Theorem 5.6]{Canizo_2023}, since Proposition \ref{prop:lyapunov} provides, in the words of those authors, a weak generator Lyapunov condition, and Corollary \ref{coroll:Doeblin-Harris} gives a Harris irreducibility condition.  

\vspace{.3cm}

For completeness, and because the argument is less redundant with the one in \cite{Bernou_2021} in this case, we provide a proof of \eqref{eq:exponential_cvg_difference}. We use the approach of \cite[Proof of Theorem 3.2]{Canizo_2023}. We start with the following lemma:

\begin{lemma}
	\label{lemma:contraction_mu}
	Under Hypothesis \ref{hypo:boundary}-\ref{hypo:sigma_below}, for all $\alpha \in (1,d)$, there exist $T > 0$, $\mu > 0$, $\gamma_0 \in (0,1)$ such that, setting
	\begin{align*}
		\vertiii{\cdot}_\mu := \|\cdot\|_{L^1} + \mu \|\cdot\|_{m_\alpha},
	\end{align*}
	one has, for all $f \in L^1_{m_\alpha}(G)$ with $\langle f \rangle = 0$, 
	\begin{align}
		\label{eq:contraction_mu_norm}
		\vertiii{S_T f}_\mu \le \gamma_0 \vertiii{f}_\mu. 
	\end{align}
\end{lemma}

\medskip

\begin{proof}[Proof of Lemma \ref{lemma:contraction_mu}.]

\textbf{Step 1: Reformulation of the Lyapunov inequality.} 
Let $t > 0$. Recall that $\alpha \in (1,d)$ is given. By Proposition \ref{prop:lyapunov}, and more precisely equation \eqref{eq:Lyapunov_expo}, we have, for any $f \in L^1_{m_\alpha}(G)$, 
\begin{align}
	\label{eq:basic_Lyap_expo}
	\| S_t f\|_{m_\alpha} + \sigma_0 \int_0^t \|S_s f\|_{m_\alpha} \, \d s \le \|f\|_{m_\alpha} + K_2 (1+t) \|f\|_{L^1}.  
\end{align}
Note that in particular, for all $s \in (0,t)$ 
\begin{align*}
	\|S_{t-s} S_s f\|_{m_\alpha} \le \|S_s f\|_{m_\alpha} + K_2 (1 + t-s) \|S_s f\|_{L^1}
\end{align*}
which rewrites 
\begin{align}
	\label{eq:m_1_traffic}
	\|S_t f\|_{m_\alpha} - K_2(1+t-s) \|S_s f\|_{L^1} \le \|S_s f\|_{m_\alpha},
\end{align}
and injecting \eqref{eq:m_1_traffic} inside \eqref{eq:basic_Lyap_expo} gives
\begin{align*}
	\|S_t f\|_{m_\alpha} + \sigma_0 \int_0^t \Big( \|S_t f \|_{m_\alpha} - K_2 (1+t-s) \|S_s f\|_{L^1} \Big) \d s \le \|f\|_{m_\alpha} + K_2(1+t) \|f\|_{L^1}. 
\end{align*}
Using also the $L^1$ contraction from Theorem \ref{thm:well_posed_trace}, we obtain
\begin{align*}
	\|S_t f\|_{m_\alpha} \le \frac{1}{1+\sigma_0 t} \|f\|_{m_\alpha} + \frac{K_2}{1 + \sigma_0 t} \big( \sigma_0 \tfrac{t^2}{2} + (1+\sigma_0) t + 1\big) \|f\|_{L^1},
\end{align*}
and ultimately, for some constant $K_3 > 0$, 
\begin{align}
	\label{eq:final_rewriting_Lyap}
	\|S_t f\|_{m_\alpha} \le \frac{1}{1+\sigma_0 t} \|f\|_{m_\alpha} + K_3 (1+t) \|f\|_{L^1}.
\end{align}

We note that the combination of \eqref{eq:final_rewriting_Lyap} and Theorem \ref{thm:Doeblin-Harris} already fits \cite[Section 3]{Canizo_2023}, so that one can readily apply their results. To facilitate the task of the reader, we nevertheless present a proof starting from those two results, in particular because our Doeblin-Harris condition, Corollary \ref{coroll:Doeblin-Harris} is slightly non-standard. 

\medskip 

\textbf{Step 2: describing two alternatives.} According to Corollary \ref{coroll:Doeblin-Harris}, for all $\rho > 2$, there exists $T(\rho) = \xi \rho$ for some constant $\xi > 0$ and a non-negative measure $\nu$ on $G$ with $\nu \not \equiv 0$, $\langle \nu \rangle \le 1$ such that 
\[ S_{T(\rho)} h \ge \nu \int_{\{(x,v) \in G, m_1(x,v) \le \rho\}} h \, \d v \, \d x, \]
for all $h \in L^1(G)$ with $h \ge 0$. 

By assumption, $f \in L^1_{m_\alpha}(G)$ and $\langle f \rangle = 0$. We set, for any $\rho > 0$, $\kappa(\rho) = K_3(1+T(\rho))$. Since $T(\rho) = \xi \rho$ for some constant $\xi > 0$, $\kappa(\rho) \underset{\rho \to \infty}{\sim} C \rho$ for some $C > 0$. Since $\alpha \in (1,d)$, one can find $\rho_0$ such that, for all $\rho > \rho_0$, $T(\rho) > 1$, $\kappa(\rho) > 1$ and $\rho^{\alpha} \ge \frac{4 \kappa(\rho)}{1-\tfrac{1}{1+\sigma_0}}$. We fix $\rho > \rho_0$, $T = T(\rho) > T(\rho_0) =: T_0$ for the remaining part of the proof. Note that, since $T(\rho) = \xi \rho$ for some constant $\xi$, any choice of $T > T(\rho_0)$ is possible. We set $A := \frac{\rho^\alpha}{4}$ and define, for $\mu > 0$ to be chosen, the $\mu$-norm by
\[ \vertiii{f}_{\mu} := \|f\|_{L^1} + \mu \|f\|_{m_\alpha}.  \]
We distinguish two cases. Indeed, we have the alternative: 
\begin{subequations}
	\begin{align}
		\|f\|_{m_\alpha} &\le A \|f\|_{L^1}, \label{eq:Step_2_1} \\
		\hbox{ or } \|f\|_{m_\alpha} &> A \|f\|_{L^1}.  \label{eq:Step_2_2}
	\end{align}
\end{subequations}

\medskip 

\textbf{Step 3: alternative \eqref{eq:Step_2_1}.} 
We prove a convergence result in the $\mu$-norm in the case of the first alternative, \eqref{eq:Step_2_1}. Set, for all $\Lambda > 0$, $D_\Lambda = \{(x,v) \in G, m_1(x,v) \le \Lambda\}$. Using $\langle f \rangle = 0$, that $m_\alpha \equiv m_1^\alpha$ and Corollary \ref{coroll:Doeblin-Harris}, we have, for all $(x,v) \in G$, 
\[ \begin{aligned}
	S_Tf_{\pm}(x,v) &\geq \nu(x,v) \int_{G} f_{\pm}(x',v') \,  \d v' \d x' - \nu(x,v) \int_{D_{\rho}^c} f_{\pm}(x',v') \, \d v' \d x' \\
	& \geq \frac{\nu(x,v)}{2} \int_G |f(x',v')| \, \d v' \d x'  - \nu(x,v) \int_{D_{\rho}^c} |f(x',v')|\, \d v' \d x' \\
	& \geq \frac{\nu(x,v)}{2} \int_G |f(x',v')| \, \d v' \d x'  - \frac{\nu(x,v)}{\rho^\alpha} \int_{G}  |f(x',v')| m_\alpha(x',v') \, \d v' \d x' \\
	& \geq \frac{\nu(x,v)}{2} \int_G |f(x',v')| \, \d v' \d x'  - \frac{\nu(x,v)}{4}  \int_{G}  |f(x',v')| \, \d v' \d x' \\
	&= \frac{\nu(x,v)}{4}  \int_{G}  |f(x',v')| \, \d v' \d x' =: \bar \nu(x,v),
\end{aligned} \]
where the third inequality is given by the fact that $D_{\rho}^c = \{(x,v) \in G, m_\alpha(x,v)/\rho^\alpha \ge 1\}$. The last inequality is obtained by condition (\ref{eq:Step_2_1}). The final equality stands for a definition of $\bar \nu(x,v)$ for all $(x,v) \in G$. Note that $\bar \nu \geq 0$ on $G$.
We deduce,
\[ \begin{aligned}
	|S_T f| &= |S_T f_+ - \bar \nu- (S_T f_- - \bar \nu)| \\
	&\leq |S_T f_+ - \bar \nu| + |S_T f_- - \bar \nu| \\
	&= S_T f_+ + S_T f_- - 2 \bar \nu = S_T|f| - 2 \bar \nu,
\end{aligned} \]
and, integrating over $G$, we have, using the contraction property, that $\bar \nu = \frac{\nu}{4} \|f\|_{L^1}$, and that $\nu$ is non-negative with $\langle \nu \rangle \le 1$,
\begin{equation}\begin{aligned}
		\label{IneqContractionDoeblin}
		\|S_T f \|_{L^1} \leq \|f\|_{L^1} - 2 \|\bar \nu\|_{L^1} = \Big( 1- \frac{\langle \nu \rangle}{2} \Big) \|f\|_{L^1} = \underline{\nu} \|f\|_{L^1},
\end{aligned}\end{equation}
with $\underline{\nu} \in (0, 1)$. Hence, $S_T$ is a strict contraction in $L^1$ in the case where $f$ satisfies \eqref{eq:Step_2_1}.
Writing $\gamma := 1/(1+\sigma_0 T) < 1$ in \eqref{eq:final_rewriting_Lyap} and using the definition of $\kappa(\rho)$, we derive an inequality on the $\mu$-norm of $S_T f$, 
\[ \begin{aligned}
	\vertiii{S_T f}_{\mu} &= \|S_T f\|_{L^1} + \mu \|S_T f\|_{m_\alpha} \\
	&\leq \underline{\nu} \|f\|_{L^1} + \mu \big( \gamma \|f\|_{m_\alpha} + \kappa(\rho) \|f\|_{L^1} \big) \\
	&\leq \big(\underline{\nu} + \mu \kappa(\rho) \big)\|f\|_{L^1} + \mu \gamma \|f\|_{m_\alpha}.
\end{aligned} \]
Finally, we choose $0 < \mu \leq \frac{1 - \underline{\nu}}{2 \kappa} < 1$ and deduce
\begin{equation}\begin{aligned}
		\label{eq:end_alternative_1}
		\vertiii{S_T f}_{\mu} \le \gamma_1 \vertiii{f}_{\mu}
\end{aligned}
\end{equation}
with $\gamma_1 := \min(\gamma \mu, \frac{1 + \underline{\nu}}{2}) < 1$.  

\medskip 

\textbf{Step 4: alternative \eqref{eq:Step_2_2}.} By choice of $T > 1$ and $\rho$ in Step 2, we have, with $\gamma$ as before, 
\begin{align}
	\label{eq:temp_rho}
	\frac{\rho^\alpha}{4 \kappa(\rho)} > \frac{1}{1 - \gamma}.
\end{align} 
By choice of $A$, a direct use of \eqref{eq:final_rewriting_Lyap} leads to
\begin{align*}
	\|S_T f\|_{m_\alpha} &\le \gamma \|f\|_{m_\alpha} + \kappa(\rho) \|f\|_{L^1} \\
	&\le \gamma \|f\|_{m_\alpha} + \frac{\kappa(\rho)}{A} \|f\|_{m_\alpha} \\
	&\le \big(\gamma + 4 \frac{\kappa(\rho)}{\rho^\alpha}\big) \|f\|_{m_\alpha} \le \tilde \gamma \|f\|_{m_\alpha}, 
\end{align*}
with $0 < \tilde \gamma := 4 \kappa(\rho)/\rho^\alpha + \gamma < 1$ by \eqref{eq:temp_rho}. Hence
\begin{align*}
	\vertiii{S_T f}_{\mu} &= \|S_T f\|_{L^1} + \mu \|S_T f\|_{m_\alpha} \\
	&\le \|f\|_{L^1} + \mu \tilde \gamma \|f\|_{m_\alpha} \\
	&\le (1 - \mu \epsilon_0) \|f\|_{L^1} + \mu(\tilde \gamma + \epsilon_0) \|f\|_{m_\alpha},
\end{align*}
where we used $m_\alpha \ge 1$ to obtain the last inequality. Using that $\tilde \gamma < 1$, we can choose $\epsilon_0 > 0$ small enough ($\mu$ is fixed by the previous step) so that
\begin{align*}
	\vertiii{S_T f}_{\mu} \le \gamma_2 \vertiii{f}_{\mu},
\end{align*}
with $\gamma_2 := \min(1- \mu \epsilon_0, \tilde \gamma + \epsilon_0) < 1$. 

\medskip 

\textbf{Step 5: conclusion.} 
We set $\gamma_0 = \max(\gamma_1, \gamma_2) < 1$ (which depends on our choice of $\alpha$) to complete the proof. 
\end{proof}

From there, the proof follows by a semigroup argument. The extension to $\alpha$ in $(0,d)$ is obtained by an interpolation argument. The key tool for this is the following corollary applicable to spaces of the form $\{f \in L^1_w(G), \langle f \rangle = 0\} $ with $w \ge 1$ some weight on $G$. We denote $\vertiii{H}_{A\to B}$ the operator norm of $H$ acting between the two Banach spaces $A$ and $B$. 

\begin{coroll}\cite[Corollary 3]{Bernou_2020}
	\label{coroll:Interpol}
	Let $\phi_1, \phi_2, \tilde{\phi}_1, \tilde{\phi}_2$ be four measurable functions on $G$ positive almost everywhere. Let also $A_1 = L^1_{\phi_1}(G)$, $A_2 = L^1_{\phi_2}(G)$, $\tilde{A}_1 = L^1_{\tilde{\phi}_1}(G)$, $\tilde{A}_2 = L^1_{\tilde{\phi}_2}(G)$. 
	Let, for all $\gamma \in (0,1)$, $\phi_{\gamma}$ and $\tilde{\phi}_{\gamma}$ be defined by
	\[ \phi_{\gamma} := \phi_1^{\gamma} \phi_2^{1-\gamma}, \qquad \tilde{\phi}_{\gamma} := \tilde{\phi}_1^{\gamma} \tilde{\phi}_2^{1-\gamma}, \]
	respectively, and $A_{\gamma} = L^1_{\phi_{\gamma}}(G)$, $\tilde{A}_{\gamma} = L^1_{\tilde{\phi_{\gamma}}}(G)$. Assume that there exists a bounded projection $\Pi: (A_i,\tilde{A}_i) \to (A_i',\tilde{A}_i')$ for $i \in \{1,2\}$ with $A_i' \subset A_i$, $\tilde{A}_i' \subset \tilde{A}_i$. Let also $A'_{\gamma} = (A'_1 + A'_2) \cap A_{\gamma}$ and $\tilde{A}'_{\gamma} = (\tilde{A}'_1 + \tilde{A}'_2) \cap \tilde{A}_{\gamma}$. Assume that $S$ is a linear operator from $A'_1$ to $\tilde{A}'_1$ and from $A'_2$ to $\tilde{A}'_2$ with
	\[ \vertiii{S}_{A'_1 \to \tilde{A}'_1} \leq N_1, \qquad \vertiii{S}_{A_2' \to \tilde{A}_2'} \leq N_2, \]
	for $N_1, N_2 > 0$. Then $S$ is a linear operator from $A'_{\gamma}$ to $\tilde{A}'_{\gamma}$ and there exists $C > 0$ depending only on $\Pi$ such that
	\[ \vertiii{S}_{A'_{\gamma} \to \tilde{A}'_{\gamma}} \leq C N_1^{\gamma} N_2^{1- \gamma}. \]
\end{coroll}

\begin{proof}[Proof of Theorem \ref{thm:main_1}.]
	\textbf{Step 1: Proof in the case $\alpha \in (1,d)$.} Let $T, \mu$ and $\gamma_0 \in (0,1)$ be given by Lemma \ref{lemma:contraction_mu}. Let $t > 0$ and write $j = \lfloor t/T \rfloor$. Then $j \in (t/T - 1, t/T]$ and using the $L^1$ contraction and \eqref{eq:contraction_mu_norm}
\begin{align*}
 \|S_t f\|_{L^1} = \|S_{t-jT} S_{jT} f\|_{L^1} \le \|S_{jT} f\|_{L^1} \le \vertiii{S_{jT} f}_{\mu} &\le e^{j \ln(\gamma_0)} \vertiii{f}_{\mu} \\
 &\le e^{-\ln(\gamma_0)} e^{t\frac{\ln(\gamma_0)}{T}} (1+\mu) \|f\|_{m_\alpha},
\end{align*}
where we used 
\[\frac1{1+\mu} \vertiii{\cdot}_\mu \le \|\cdot\|_{m_\alpha} \le \frac{1}{\mu} \vertiii{\cdot}_\mu. \]
Thus, for some constant $C > 0$ (possibly larger than $e^{-\ln(\gamma_0)} (1+\mu)$ to handle the case $t < T$), for $\kappa = -\ln(\gamma_0)/T > 0$, we obtain 
\begin{align}
	\label{eq:conclusion_Step_5}
	\|S_t f\|_{L^1} \le Ce^{-\kappa t} \|f\|_{m_{\alpha}}. 
\end{align}

\medskip 

\textbf{Step 2: Interpolation.} We derive a convergence result for $\|\cdot\|_{m_q}$ for $q \in (0,1]$. Set 
\[ L^1_0(G) = \{g \in L^1(G), \langle g \rangle = 0\}  \hbox{ and }  L^1_{w,0}(G) = \{g \in L^1_w(G), \langle g \rangle = 0 \} \]
for any weight $w$ on $\bar{G}$. We recall the notation $M_1$ from \eqref{eq:def_M_1}.
Note that $\int_{\R^d} |v|^2 \, M_1(v) \, \d v = 1$.
We consider $\Pi: L^1(G) \to L^1_0(G)$ the bounded projection such that, for all $h \in L^1(G)$, $(x,v) \in G$, 
\[ \Pi h(x,v) = h(x,v) - \frac{M_1(v)|v|^2}{|\Omega|} \int_G h(y,w) \, \d y \, \d w, \]
where $|\Omega|$ denotes the volume of $\Omega$. By use of hyperspherical coordinates, it is straightforward to check that $\Pi h \in L^1_{m_{\frac{3}{2}}}(G)$ for all $h \in L^1_{m_{\frac{3}{2}}}(G)$. Also, there exists a constant $C_\Pi > 0$ such that $\|\Pi h \|_{m_{\frac{3}{2}}} \le C_\Pi \|h\|_{m_{\frac{3}{2}}}$ for all $h \in L^1_{m_{\frac{3}{2}}}(G)$ and $\|\Pi h \|_{L^1} \le C_\Pi \|h\|_{L^1}$. Since $\langle h \rangle = 0$ implies $\Pi h = h$, and $\langle \Pi h \rangle = 0$ for all $h \in L^1(G)$, $\Pi$ is a bounded projection as claimed. 

Let $t > 0$. From Theorem \ref{thm:well_posed_trace}, we have
\begin{align*}
	\vertiii{S_t}_{L^1_0(G) \to L^1_0(G)} \le 1, 
\end{align*}
and from Step 1., since $3/2 \in (1,d)$,  there exist $C, \kappa > 0$ such that 
\[ \vertiii{S_t}_{L^1_{m_{\frac{3}{2}},0}(G) \to L^1_0(G)} \le Ce^{-\kappa t}. \]
We apply Corollary \ref{coroll:Interpol} with the projection $\Pi$ and the values:
\begin{enumerate}[1.]
	\item $A_1 = L^1_{m_{\frac32}}(G)$,  $\tilde{A}_1 = L^1(G)$, and, using the definition of $\Pi$, $A'_1 = L^1_{m_{\frac32}, 0}(G)$, $\tilde{A}'_1 = L^1_{0}(G)$,
	\item $A_2 = \tilde{A}_2 = L^1(G)$, and, using the definition of $\Pi$, $A_2' = \tilde{A}_2' = L^1_0(G)$,
	\item $\gamma = \tfrac{2q}{3} \in (0,1)$, so that $A_{\gamma} = L^1_{m_q}(G)$,  $\tilde{A}_{\gamma} = L^1(G)$, and, using the definition of $\Pi$, we have $A_{\gamma}' = (A_1' + A_2') \cap A_{\gamma} = L^1_{m_q, 0}(G)$, $\tilde{A}_\gamma = (\tilde{A}_1' + \tilde{A}_2') \cap \tilde{A}_\gamma = L^1_{0}(G)$. 
\end{enumerate}
We conclude from the corollary that
\begin{align*}
	\vertiii{S_t}_{ L^1_{m_q, 0}(G) \to L^1_0(G)} \le C^{2q/3} e^{-\frac{2q \kappa}{3} t}.
\end{align*}
The same argument can be applied for all $t \ge 0$. The conclusion follows.
\end{proof}
%
%

\subsection{Proof of Theorem \ref{thm:main_2}}

We first note that the proofs of \eqref{eq:cvg_polynomial} and \eqref{eq:cvg_expo} are straightforward applications of Theorem \ref{thm:main_1} once \textit{i.} is established. Thus, \textit{i.} is the sole point of the statement whose proof is lacking.

\medskip

As before, we only detail the exponential case: the polynomial one can be established exactly as in \cite[Section 5.3]{Bernou_2021}. 

\medskip 

\textbf{Step 1: Uniqueness.} Let $\epsilon \in (0,1/2)$, $\alpha = d - \epsilon$. We will recycle Lemma \ref{lemma:contraction_mu}. Assume there exists two steady states $f_\infty$, $g_\infty$ with the desired properties. Applying Lemma \ref{lemma:contraction_mu} with $\alpha$ gives the existence of $T, \mu > 0$ and $\gamma_0 \in (0,1)$ such that \eqref{eq:contraction_mu_norm} holds for all $f \in L^1_{m_\alpha, 0}(G)$. Since both $f_\infty$ and $g_\infty$ belong to $L^1_{m_{\alpha}}(G)$, with $\langle f_\infty- g_\infty \rangle =0$ by linearity, we obtain
\begin{align}
	\label{eq:temp_uniqueness}
	\vertiii{S_T(f_\infty - g_\infty)}_\mu \le \gamma_0 \vertiii{f_\infty - g_\infty}_\mu.
\end{align}
Since $S_T (f_\infty - g_\infty) = f_\infty - g_\infty$ by linearity and since both are steady states, \eqref{eq:temp_uniqueness} rewrites
\begin{align*}
	\vertiii{f_\infty - g_\infty}_\mu \le \gamma_0 \vertiii{f_\infty - g_\infty}_\mu. 
\end{align*}
It follows that $\vertiii{f_\infty - g_\infty}_\mu = 0$, and thus $\|f_\infty - g_\infty\|_{L^1} = 0$, which proves the uniqueness.

\medskip 

\textbf{Step 2: Existence.} Set $\alpha = d - \epsilon$, let $g \in L^1_{m_\alpha}(G)$ with $\langle g \rangle = 1$ and let again $T, \mu > 0$ and $\gamma_0 \in (0,1)$ given by Lemma \ref{lemma:contraction_mu}. Define, for all $h \ge 1$, 
\[ g_h = S_{Th} g, \qquad f_h = g_{h+1}-g_h. \]
Note that for all $h \ge 1$, $\langle f_h \rangle = 0$ by mass conservation. The inequality \eqref{eq:contraction_mu_norm} applied to $f_h$ reads
\begin{align}
	\label{eq:contraction_fh}
	\vertiii{f_{h+1}}_\mu \le \gamma_0 \vertiii{f_h}_\mu. 
\end{align}
It follows that $(\vertiii{f_h}_\mu)_{h \in \mathbb{N}^*}$ is a non-negative, decreasing sequence converging towards $0$. Hence, for $0 < \omega \ll 1$ fixed, one can set $N > 0$ such that for all $ r > N$,
\[ \vertiii{f_r}_\mu \le \frac{\mu}{\gamma_0} (1-\gamma_0) \omega. \]
Next, recalling $\|\cdot\|_{m_\alpha} \le \frac1{\mu} \vertiii{\cdot}_\mu$, we have, for $q > r > N$, 
\begin{align*}
	\mu \|g_{q+1} - g_{r+1}\|_{m_\alpha} &= \mu \Big\| \sum_{h = r+1}^q f_h\Big\|_{m_\alpha} \\
	&\le \mu \sum_{h = r}^{q-1} \|S_T f_h\|_{m_\alpha} \\
	&\le  \sum_{h = r}^{q-1} \vertiii{S_T f_h}_\mu \\
	&\le  \vertiii{f_r}_\mu \sum_{h=1}^{q-r} \gamma_0^h \\
	&\le \vertiii{f_r}_\mu \frac{\gamma_0}{1-\gamma_0} \le \mu \omega,
\end{align*}
by definition of $N$, where we used repeatedly \eqref{eq:contraction_fh}. We deduce that $(g_h)_{h \ge 0}$ is a Cauchy sequence in the Banach space $L^1_{m_\alpha}(G)$, and thus converges towards a limit $f_\infty$ with $\langle f_\infty \rangle = \langle g \rangle$ by mass conservation. A similar argument to the one used in the proof of uniqueness shows that $f_\infty$ is independent of the starting function $g \in L^1_{m_\alpha}(G)$.

\section{Counter-example for Hypothesis \ref{hypo:sigma_below} and lower bounds on the convergence rate}
\label{sec:counter_ex}

This section is devoted to the proof of Theorem \ref{thm:lower_bounds}. We draw inspiration from the work of Aoki and Golse \cite[Section 3]{Aoki_2011}.

We let $B := B(0,1)$ be the unit ball in $\R^d$ centered at $0$ and use $|B|$ to denote its volume. Upon translating and rescaling, we may assume $0 \in \Omega$,  and $R = \frac{1}{2} d(\partial \Omega,0) > 1$. In the whole section, we pick the following initial data: for $0 < \epsilon \ll 1$ to be chosen, $(x,v) \in G$,
\begin{align}
	\label{eq:def_f_sec_upper}
	 f(x,v) = \frac{1}{\epsilon^{2d} |B|^2 } \mathbf{1}_{\epsilon B}(x) \mathbf{1}_{\epsilon B} (v). 
	\end{align}
Note that $f \in L^1_{m_\alpha}(G)$ for all $\alpha \in (0,d)$ and that $\langle f \rangle = 1$.
Throughout the proof, $f_\infty$ is given by Theorem \ref{thm:main_2}, and we set $H_0 = \|f_\infty\|_{L^\infty(G)}$,
which is finite by assumption. We start by establishing a preliminary lemma, deduced from bounds for the convergence of $(S_t f)_{t \ge 0}$ towards $f_\infty$. This leads to an inequality parameterized by $\epsilon$. We then deduce all three results of Theorem \ref{thm:lower_bounds} by making different choices of $\epsilon$ in the various settings.

\medskip

In the whole section, we write $|B|$ for the volume of $B(0,1)$, $x^+$ denotes $\max(0,x)$ for $x \in \R$ and for all $A \subset \R^d$, $u \in \R^d$, $A+u = \{z+u: z \in A\}$. 

\subsection{A preliminary lemma}

We prove the following:
\begin{lemma}
	Let $\alpha \in (0,d)$. Assume there exists a uniform decay rate $E: \R_+ \to \R_+$ such that $E(t) \to 0$ as $t \to \infty$ and for all $g \in L^1_{m_\alpha}(G)$ with $\langle g \rangle = 1$, for all $t \ge 0$,
	\begin{align*}
		\Big\|S_{t}g - f_\infty\|_{L^1(G)} \le E(t) \|g - f_\infty\|_{m_\alpha}.
	\end{align*}
	Then, there exist $C_\alpha > 0$ allowed to depend on $\|f_\infty\|_{m_\alpha}$ and $\epsilon_0 \in (0,1)$ such that for all $\epsilon \in (0,\epsilon_0)$,
	\begin{align}
		\label{eq:lemma_lower_bound}
		&\int_G \frac{1}{\epsilon^{2d} |B|^2} \mathbf{1}_{\epsilon B}(v) \mathbf{1}_{\epsilon B + tv} (x)  \mathbf{1}_{\{0 \le t|v| \le R - \epsilon\}}\nonumber  \\
		&\qquad \times \Big[e^{-\int_0^{t} \sigma(x-(t-u)v) \, \d u}  - \epsilon^{2d}|B|^2 H_0 \Big]^+ \, \d v \, \d x \le C_\alpha E(t) \epsilon^{-\alpha}. 
	\end{align}
\end{lemma}
 
 \begin{proof} We choose $g = f$ with $f$ given by \eqref{eq:def_f_sec_upper} in the whole proof. 
 	
 	\medskip 
 
 \textbf{Step 1: comparison principle.} We introduce the problem 
\begin{align}
	\label{eq:comparison_counter_ex}
	\left\{
	\begin{array}{lll}
		&\partial_t \Phi + v \cdot \nabla \Phi = - \sigma(x)  \Phi \qquad &\mathrm{in } \quad \R_+ \times G, \\
		&\gamma_- \Phi = 0 \qquad &\mathrm{on } \quad \R_+ \times \Sigma_-, \\
		&\Phi_{| t = 0} = f, \qquad &\mathrm{in } \quad G, 
	\end{array}
	\right.
\end{align}
which corresponds to the evolution problem for the density of particles killed when a clock parameterized by $\sigma$ rings, and when they reach the boundary. By the methods of characteristics (see also Step 2 of the proof of \eqref{eq:Duhamel}), we have that, for $(t,x,v) \in \R_+ \times G$,
\[ \Phi(t,x,v) = \mathbf{1}_{\{\tau(x,-v) \ge t\}} e^{-\int_0^t \sigma(x-(t-s)v) \, \d s} f(x-tv,v) \]
is the unique solution with $\Phi$ in $L^\infty([0,\infty); L^1(G))$.
Using the Duhamel principle \eqref{eq:Duhamel} and the fact that $(S_t)_{t \ge 0}$ is non-negative,
\begin{align*}
	S_t f(x,v) \ge \Phi(t,x,v). 
\end{align*}

\medskip 

\textbf{Step 2: a lower bound for the left-hand-side of \eqref{eq:counter_ex}.} We have, by monotonicity of $x \mapsto x^+$, for $t > 0$
\begin{align*}
	&\Big\| S_{t} f - f_\infty \Big\|_{L^1(G)} \\
	&\ge \int_G \Big(S_{t} f(x,v) - f_\infty(x,v) \Big)^+ \, \d v \, \d x \\
	&\ge \Big( \Phi(t,x,v) - f_\infty(x,v) \Big)^+ \, \d v \, \d x  \\
	&\ge \int_G \Big[\frac{1}{\epsilon^{2d}|B|^2} e^{-\int_0^{t} \sigma(x-(t-u)v) \, \d u} \mathbf{1}_{\epsilon B}(v) \mathbf{1}_{\epsilon B + tv} (x)  \mathbf{1}_{\{\tau(x,-v) \ge t\}} - H_0 \Big]^+ \, \d v \, \d x \\
	&\ge \int_G \Big[\frac{1}{\epsilon^{2d}|B|^2} e^{-\int_0^{t} \sigma(x-(t-u)v) \, \d u} \mathbf{1}_{\epsilon B}(v) \mathbf{1}_{\epsilon B + t v} (x)  \mathbf{1}_{\{0 \le t |v| \le R - \epsilon\}} - H_0 \Big]^+ \, \d v \, \d x 
\end{align*}
where we have used that, on $\{t\le \tau(x,-v)\}$, 
\[\tau(x,-v) = t + \tau(x-tv, -v) \ge (2R - \epsilon)/|v|\] if $x - t v \in \epsilon B$ by definition of $R$. Using the properties of $x \mapsto x^+$, we then have
\begin{align*}
		\Big\| S_{t} f - f_\infty \Big\|_{L^1(G)} 
		&\ge \int_G \frac{1}{\epsilon^{2d} |B|^2} \mathbf{1}_{\epsilon B}(v) \mathbf{1}_{\epsilon B + tv} (x)  \mathbf{1}_{\{0 \le t|v| \le R - \epsilon\}} \\
		&\qquad \times \Big[e^{-\int_0^{t} \sigma(x-(t-u)v) \, \d u}  - \epsilon^{2d}|B|^2 H_0 \Big]^+ \, \d v \, \d x. 
\end{align*}

\medskip 

\textbf{Step 3: An upper bound for the right-hand-side.} For $C_\alpha > 0$ a constant depending on $\alpha$, $\|f_\infty\|_{m_\alpha}$ allowed to change from line to line, using that for all $(x,v) \in G$, by convexity,
\begin{align*}
	 m_\alpha(x,v) \le C_\alpha \Big( 1 + \frac{1}{|v|^\alpha} + |v|^{2\delta \alpha}\Big), 
	\end{align*}
	
	\vspace{-.2cm}
	
we get
\begin{align*}
	E(t) \|f - f_\infty\|_{m_\alpha} &\le E(t) \Big[ \int_{\epsilon B} \int_{\epsilon B} \frac{1}{\epsilon^{2d}|B|^2} \Big(e^2 + \frac{d(\Omega)}{c_4|v|} + |v|^{2 \delta}\Big)^\alpha + \|f_\infty\|_{m_\alpha} \Big] \\
	&\le C_\alpha E(t) \Big[ 1 + \epsilon^{2\delta \alpha} + \frac{1}{\epsilon^\alpha} \Big].
\end{align*}
Indeed, note that by use of hyperspherical coordinates, since $\alpha < d$
\begin{align*}
	\int_{\epsilon B} \frac{1}{\epsilon^d |B|} \frac{d(\Omega)^\alpha}{c_4^\alpha|v|^\alpha} \, \d v \le C_\alpha \int_0^\epsilon \frac1{\epsilon^d} r^{d-1-\alpha} \, \d r \le C_\alpha \epsilon^{-\alpha}.
\end{align*}
Choosing $\epsilon_0 = 2^{-1 - \frac1{\alpha}} < 1$ concludes the proof. 
\end{proof}

The next subsections are devoted to the proof of Theorem \ref{thm:lower_bounds}.

\subsection{Proof of point (1) of Theorem \ref{thm:lower_bounds}} \

\medskip 

\textbf{Step 1: improved lower bound.} Starting from \eqref{eq:lemma_lower_bound}, we note first that, since $\sigma \le \sigma_\infty$ and by monotonicity of $x \mapsto x^+$, one obtains
\begin{align*}
&\Big[e^{- \sigma_\infty t}  - \epsilon^{2d}|B|^2 H_0 \Big]^+ \int_G \frac{1}{\epsilon^{2d} |B|^2} \mathbf{1}_{\epsilon B}(v) \mathbf{1}_{\epsilon B + tv} (x)  \mathbf{1}_{\{0 \le t|v| \le R - \epsilon\}}  \d v \, \d x \le C_\alpha E(t) \epsilon^{-\alpha}.
\end{align*} 
Using Tonelli's theorem to perform first the integration in space, we find
\begin{align*}
	&\Big[e^{- \sigma_\infty t}  - \epsilon^{2d}|B|^2 H_0 \Big]^+ \int_{\R^d} \frac{1}{\epsilon^{d} |B|} \mathbf{1}_{\epsilon B}(v)  \mathbf{1}_{\{0 \le t|v| \le R - \epsilon\}}  \d v \, \d x \le C_\alpha E(t) \epsilon^{-\alpha}.
\end{align*} 
The integral in $v$ lies in the domain $\{v \in \R^d: |v|\le \epsilon \hbox{ and } |v| \le \frac{R-\epsilon}{t} \}$ so one gets
\begin{align*}
	&\Big[e^{-\sigma_\infty t}  - \epsilon^{2d}|B|^2 H_0 \Big]^+ \frac{1}{\epsilon^d} \min\Big( \epsilon^d, \frac{(R-\epsilon)^d}{t^d} \Big)  \le C_\alpha E(t) \epsilon^{-\alpha},
\end{align*}
and we conclude that 
\begin{align}
	\label{eq:final_step_pt_1_thm_lower}
	&\Big[e^{-\sigma_\infty t}  - \epsilon^{2d}|B|^2 H_0 \Big]^+ \min\Big( 1, \Big(\frac{R-\epsilon}{\epsilon t} \Big)^d \Big)  \le C_\alpha E(t) \epsilon^{-\alpha}.
\end{align}

\medskip 

\textbf{Step 2: conclusion.} With the choice $\epsilon = e^{-\sigma_\infty t}$, one gets from \eqref{eq:final_step_pt_1_thm_lower}, for $t$ large enough so that $\epsilon < \epsilon_0$, $e^{-\sigma_\infty t} - e^{-2d\sigma_\infty t} |B|^2 H_0 \ge e^{-\sigma_\infty t}/2$ and $\frac{e^{t\sigma_\infty} R}{t} - \frac1{t}  > 2$,
\begin{align*}
	\Big[e^{-\sigma_\infty t} - e^{-2d \sigma_\infty t} |B|^2 H_0 \Big]^+ \le C_\alpha e^{\sigma_\infty \alpha t} E(t). 
\end{align*}
We conclude that, for $t$ large enough,
\begin{align*}
	E(t) \ge C_\alpha e^{-\sigma_\infty (1+\alpha) t}.
\end{align*}

\subsection{Proof of point (2) and (3) of Theorem \ref{thm:lower_bounds}} \

We rewrite slightly differently the previous setting. Once again, this is done without loss of generality by rescaling and translating, in view of the hypotheses. 
We assume that $0 \in \Omega$, that $R = \frac{1}{2} d(\partial \Omega,0) > 1$, and that $\sigma \equiv 0$ on $B(0,1)$. Note that $\sigma$ may cancel on a larger region of $\Omega$, but we can always reduce the considered ball to fit this framework. We consider again the initial data $f$ given by \eqref{eq:def_f_sec_upper}.  

\medskip

\textbf{Step 1: a lower bound for the killing term.} Let $v \in \epsilon B$, $t > 0$, $x \in \Omega$ with $x \in \epsilon B + t v$. Then, for all $u \in [0,t]$, by assumptions on $\sigma$ 
	\begin{align*}
		\sigma(x - (t-u)v ) = \sigma(x-(t-u)v) \mathbf{1}_{\{x - (t-u)v \not \in B(0,1)\}} \le \sigma_\infty \mathbf{1}_{\{x - (t-u)v \not \in B(0,1)\}}. 
	\end{align*}
	Moreover, on $\{u: x - (t-u)v \not \in B(0,1)\}$, 
	\begin{align*}
		|uv| \ge |x - (t-u)v| - |x - tv| \ge 1 - \epsilon
	\end{align*}
	and since $|v| \le \epsilon$, we find $u \ge \frac{1-\epsilon}{\epsilon}$. Hence 
	\begin{align*}
		\Big\{u \in [0,t]: x - (t-u)v \not \in B(0,1) \Big\} \subset \Big\{u \in [0,t]: u \ge \frac{1-\epsilon}{\epsilon} \Big\}.
	\end{align*} We thus get
	\begin{align*}
		e^{-\int_0^{t} \sigma(x-(t-u)v) \, \d u} \ge e^{-\sigma_\infty \big(t-\frac{1-\epsilon}{\epsilon}\big)^+}. 
	\end{align*}
	
	\medskip 
	
	\textbf{Step 2: improved version of \eqref{eq:lemma_lower_bound}.}
	Injecting Step 1 into \eqref{eq:lemma_lower_bound}, we find by monotonicity of $x \mapsto x^+$
	\begin{align*}
		&\Big[e^{-\sigma_\infty \big(t-\frac1{\epsilon} + 1\big)^+}  - \epsilon^{2d}|B|^2H_0 \Big]^+ \int_G \frac{ \mathbf{1}_{\epsilon B}(v)}{\epsilon^{2d} |B|^2} \mathbf{1}_{\epsilon B + tv} (x)  \mathbf{1}_{\{0 \le t|v| \le R - \epsilon\}} \, \d v \, \d x \le C_\alpha E(t) \epsilon^{-\alpha}.
	\end{align*}
	Computing the integration in space, since the integral in $v$ lies in $\{v \in \R^d: |v| \le \epsilon \hbox{ and } |v| \le \frac{R-\epsilon}{t}\}$, we get
	\begin{align*}
		&\Big[e^{-\sigma_\infty \big(t-\frac1{\epsilon} + 1\big)^+}  - \epsilon^{2d}|B|^2H_0 \Big]^+ \frac1{\epsilon^d} \min \Big( \epsilon^d, \frac{(R-\epsilon)^d}{t^d} \Big)  \le C_\alpha E(t) \epsilon^{-\alpha},
	\end{align*}
	which leads to
	\begin{align}
		\label{eq:step_2_lower_bounds}
		 \Big[e^{-\sigma_\infty \big(t-\frac1{\epsilon} + 1\big)^+}  - \epsilon^{2d}|B|^2H_0 \Big]^+ \min \Big( 1, \frac{(R-\epsilon)}{\epsilon t} \Big)^d \le C_\alpha E(t) \epsilon^{-\alpha}.
	\end{align}
	
	\medskip 
	
	\textbf{Step 3: Conclusion for point (2).} Assume that $E(t) = C e^{-\kappa t}$. From Step 2, the following inequality should check
	\begin{align}
		\label{eq:temp_cvg_rate}
		 \Big[e^{-\sigma_\infty \big(t-\frac1{\epsilon} + 1\big)^+}  - \epsilon^{2d}|B|^2H_0 \Big]^+ \min \Big( 1, \frac{(R-\epsilon)}{\epsilon t} \Big)^d \le C_\alpha e^{-\kappa t} \epsilon^{-\alpha}.
	\end{align}
	 With the choice $\epsilon = \frac{1}{t}$ and $t$ large enough so that $\epsilon < \epsilon_0$, $e^{-\sigma_\infty} - \frac{1}{t^{2d}} |B|^2 H_0 \ge \frac{e^{-\sigma \infty}}{2}$ and $R - \frac1{t} \ge 1$, we get,  using also $x^+ \ge x \in \R$, 
	\begin{align*}
		\frac{e^{-\sigma_\infty}}{2}  \le C_\alpha e^{-\kappa t} t^{\alpha}.
	\end{align*}
	Since the right-hand-side tends to $0$ as $t \to \infty$, \eqref{eq:counter_ex} can not hold.  
	
	\medskip 
	
	\textbf{Step 4: Conclusion for point (3).} Starting again from \eqref{eq:step_2_lower_bounds} and choosing $\epsilon = \frac{1}{t+1}$ for $t$ large enough so that $\epsilon < \epsilon_0$, $\frac{|B|^2 H_0}{ (t+1)^{2d}} < \frac12$ and $\frac{t+1}{t} (R - \frac1{t+1}) > 1$, one finds
	\begin{align*}
		\frac12 \le C_\alpha E(t) (t+1)^{\alpha}. 
	\end{align*}
	It easily follows that $E(t) \ge C_\alpha (t+1)^{-\alpha}$ for all $t$ large enough.

\bigskip

\textbf{Acknowledgments:} The author thanks Silvia Lorenzani for useful discussions and for pointing out the references \cite{Nguyen_2020, Yamaguchi_2016}, and Isabelle Tristani for suggesting the study of heavy-tailed kernels and the reference \cite{Cesbron_Mellet_Puel_2019}. This project received funding from the European Union’s Horizon 2020 Research and Innovation Program under the Marie Sklodowska-Curie Grant Agreement No. 10103324.


\bibliographystyle{plain}
\bibliography{biblio_DGCL}

\end{document}